\documentclass[11pt]{amsart}
\usepackage{ amsmath, amsthm, amsfonts, hyperref, graphicx, ifpdf}
\usepackage{mathrsfs,epsfig}
\usepackage[dvipsnames,usenames]{color}
\usepackage{mathrsfs}
\usepackage{amssymb}
\usepackage{faktor}
\usepackage{bm}
\usepackage{enumerate}
\usepackage[utf8]{inputenc}
\usepackage{mathtools}
\usepackage{esint}
\usepackage{tikz-cd} 
\usepackage{nicematrix}

\hypersetup{
   unicode=false,          
   pdftoolbar=true,        
   pdfmenubar=true,        
   pdffitwindow=false,     
   pdfstartview={},    
   pdftitle={},    
   pdfauthor={},     
   pdfsubject={},   
   pdfcreator={},   
   pdfproducer={}, 
   pdfkeywords={}, 
   pdfnewwindow=true,      
   colorlinks=true,       
   linkcolor=blue,          
   citecolor=blue,        
   filecolor=blue,      
   urlcolor=blue          
}

\newtheorem{theorem}{Theorem}[section]
\newtheorem{cor}[theorem]{Corollary}
\newtheorem{lem}[theorem]{Lemma}
\newtheorem{pro}[theorem]{Proposition}
\newtheorem{remark}[theorem]{Remark}

\newtheorem{Conj}[theorem]{Conjecture}
\theoremstyle{definition}

\DeclareMathOperator{\PSL}{PSL(2,\C)}

\newcommand{\af}{almost Fuchsian}
\newcommand{\afm}{almost Fuchsian manifold}

\newcommand{\cmc}{constant mean curvature}

\newcommand{\eus}{Euclidean space}
\newcommand{\ee}{evolution equation}
\newcommand{\es}{evolving surface}

\newcommand{\Fm}{Fuchsian manifold}

\newcommand{\htm}{hyperbolic three-manifold}
\newcommand{\hym}{hyperbolic metric}

\newcommand{\kg}{Kleinian group}

\newcommand{\mc}{mean curvature}
\newcommand{\mcf}{mean curvature flow}
\newcommand{\mmcf}{modified mean curvature flow}
\newcommand{\maxp}{maximum principle}
\newcommand{\ms}{minimal surface}

\newcommand{\nv}{normal vector}

\newcommand{\pc}{principal curvature}
\newcommand{\qf}{quasi-Fuchsian}

\newcommand{\sff}{second fundamental form}

\newcommand{\tg}{totally geodesic}

\def\<{\langle}
\def\>{\rangle}
\def\({\left(}
\def\){\right)}

\newcommand{\tm}{three-manifold}

\newcommand{\is}{incompressible surface}

\def\p{\partial}

\newcommand{\ksg}{Kleinian surface group}

\newcommand{\qfm}{quasi-Fuchsian manifold}

\newcommand{\be}{\begin{equation}}
\newcommand{\ene}{\end{equation}}
\newcommand{\br}{\begin{remark}}
\newcommand{\er}{\end{remark}}
\newcommand{\bl}{\begin{lem}}
\newcommand{\el}{\end{lem}}
\newcommand{\bcor}{\begin{cor}}
\newcommand{\ecor}{\end{cor}}
\newcommand{\bpro}{\begin{pro}}
\newcommand{\epro}{\end{pro}}
\newcommand{\ben}{\begin{enumerate}}
\newcommand{\een}{\end{enumerate}}
\newcommand{\bp}{\begin{proof}}
\newcommand{\ep}{\end{proof}}
\newcommand{\bpo}{\begin{pro}}
\newcommand{\epo}{\end{pro}}
\newcommand{\beq}{\begin{equation*}}
\newcommand{\eeq}{\end{equation*}}
\newcommand{\bear}{\begin{eqnarray}}
\newcommand{\eear}{\end{eqnarray}}
\newcommand{\beqar}{\begin{eqnarray*}}
\newcommand{\eeqar}{\end{eqnarray*}}
\newcommand{\bt}{\begin{theorem}}
\newcommand{\et}{\end{theorem}}

\newcommand{\vnu}{\vec{\nu}}

\newcommand{\HH}{\mathbb{H}^3}

\newcommand{\R}{\mathbb{R}}
\newcommand{\C}{\mathbb{C}}
\newcommand{\AF}{\mathcal{AF}}
\newcommand{\QF}{\mathcal{QF}}

\renewcommand{\H}{\mathbb{H}}

\newcommand{\ppl}[2]{\frac{\partial{#1}}{\partial{#2}}}

\numberwithin{equation}{section}

\allowdisplaybreaks

\numberwithin{equation}{section}

\allowdisplaybreaks

\def\XXint#1#2#3{{\setbox0=\hbox{$#1{#2#3}{\int}$}
    \vcenter{\hbox{$#2#3$}}\kern-.5\wd0}}

\makeatletter
\def\@citestyle{\m@th\upshape\mdseries}
\def\citeform#1{{\bfseries#1}}
\def\@cite#1#2{{%
  \@citestyle[\citeform{#1}\if@tempswa, #2\fi]}}
\@ifundefined{cite }{%
  \expandafter\let\csname cite \endcsname\cite
  \edef\cite{\@nx\protect\@xp\@nx\csname cite \endcsname}%
}{}
\makeatother

\begin{document}

\title[MMCF and CMC foliation in almost Fuchsian manifolds]{Modified mean curvature flow and CMC foliation conjecture in almost Fuchsian manifolds}

\author{Zheng Huang}
\address[Z. ~H.]{Department of Mathematics, The City University of New York, Staten Island, NY 10314, USA}
\address{The Graduate Center, The City University of New York, 365 Fifth Ave., New York, NY 10016, USA}
\email{zheng.huang@csi.cuny.edu}

\author{Longzhi Lin}
\address[L.~L.]{Mathematics Department\\University of California, Santa Cruz\\1156 High Street\\
Santa Cruz, CA 95064\\USA}
\email{lzlin@ucsc.edu}

\author{Zhou Zhang}
\address[Z. ~Z.]{The School of Mathematics and Statistics, The University of Sydney, NSW 2006, Australia}
\email{zhangou@maths.usyd.edu.au}

\date{}


\begin{abstract}
There has been a conjecture, often attributed to Thurston, which asserts that every {\af} manifold is foliated by closed incompressible constant mean curvature (CMC) surfaces. In this paper, for a certain class of {\afm}s, we prove the long-time existence and convergence of the {\mmcf}
$$
\ppl{F}{t}=-(H-c)\vnu,
$$
which was first introduced by Xiao and the second named author in \cite{LX12}. As an application, we confirm Thurston's CMC foliation conjecture for such a subclass of almost Fuchsian manifolds.
\end{abstract}

\maketitle

\tableofcontents
\section {Introduction}

\subsection{Motivating Questions}
	
A {\kg} $\Gamma$ is a discrete subgroup of $\PSL$, the orientation preserving isometry group of $\HH$. Any complete {\htm} can be written as some 
$\HH /  \Gamma$. When $\Gamma$ is the fundamental group of a closed oriented surface $S$ with genus $g\geq 2$, the resulting {\htm} is of the type $S \times \R$. There is a 
well-developed deformation theory for complete {\htm}s of the type $S \times \R$ (see for instance 
\cite{Thu86, BB04, Min10, BCM12} and many others). In this class of {\htm}s, regular ones (or when the limit set of $\Gamma$, denoted as $\Lambda_\Gamma$, is a Jordan curve on 
the sphere at the infinity) are {\it \qf}. In fact, the {\qf} groups form the closure of all deformation spaces of {\ksg}s.  When $\Lambda_{\Gamma}$ is actually a circle, we call $\Gamma$ a {\it Fuchsian group} and the corresponding $M$ a {\it{Fuchsian manifold}}.

It is of fundamental importance to study incompressible surfaces in {\tm} theory in order to understand the decompositions. Thurston observed that a closed surface of {\pc}s less than $1$ in magnitude (a.k.a. small curvatures) is incompressible in a {\htm} and this was proved in \cite {Lei06}. In the setting of {\htm}s which are 
diffeomorphic to $\Sigma \times \R$ (where $\Sigma$ a closed oriented surface of genus at least two), Uhlenbeck (\cite{Uhl83}) initiated a vast program using this concept to 
study incompressible {\ms}s in {\qfm}s with applications to the study of a parametrization of the space of surface group representations in $\PSL$. Closed surfaces of small curvatures, especially when they are also minimal, play an important role (see for instance \cite{Rub05, KS07, HL21, CMN22, HLS23}).

We denote the moduli space of {\qf} manifolds that is called {\qf} space by $\QF(\Sigma)$, and the {\af} space, consisting of elements of {\qf} that admit a closed {\it {\ms}} homeomorphic to $\Sigma$ of {\pc}s between $-1$ and $1$, by $\AF (\Sigma)$. It is a standard fact in geometric analysis (\cite{SY79, SU82}) that any {\qfm} admits at least one closed incompressible {\ms}, while any {\afm} admits exactly one closed {\ms} (\cite{Uhl83}). Since then these points of view set off an extensive work towards finding special incompressible surfaces in {\qfm}s, and more particularly {\afm}s. These special surfaces are often {\ms}s, or {\cmc} (i.e. {\bf CMC} in short) surfaces, or surfaces of constant Gaussian curvature. For instance, Mazzeo and Pacard proved the existence and uniqueness of CMC foliations for the ends of any {\qfm} (\cite{MP11}). Such a foliation can't be global for some {\qfm}s (\cite{HW13}). It has been a long-standing conjecture of Thurston which asserts that:
 \begin{Conj}[Thurston]\label{Thurston-Conj}
 Any {\afm} is globally, monotonically, uniquely foliated by closed oriented incompressible surfaces of {\cmc}. 
\end{Conj}
 Note that when the three-manifold is Fuchsian, namely, it is of the type $\Sigma \times \R$ and admits a unique closed totally geodesic surface $\Sigma$, then the equidistant surfaces from $\Sigma$ form a global monotone foliation of umbilic (particularly, CMC) surfaces. This conjecture remains a challenging question in the field. Recently the global {CMC} foliation on any Fuchsian manifold was extended to a nearby {\afm} in $\AF(\Sigma)$ (\cite{CMS23}) using the implicit function theorem (and the smallness of the perturbation depends on the particular Fuchsian manifold).

 In recent decades, geometric curvature flows were introduced to geometric analysis and they become a powerful tool to find such special surfaces. Thurston suggested an approach by geometric flows similar to the work of Huisken-Yau on finding foliations of {\cmc} outside a compact region in an asymptotic Euclidean space (\cite{HY96}). However in a manifold of negative curvature, its rich and expanding geometric structures and the contracting nature of some geometric 
 flows do not always agree. In this paper, we use the \textit{\mmcf} ({\bf MMCF} in short)
 \begin{equation}
\ppl{F}{t}=-(H-c)\vnu,
 \end{equation}
  to construct closed imcompressible surfaces of constant mean curvature in a certain class of {\af} manifolds, see Theorem \ref{main}. This {MMCF} is the negative gradient flow of certain functional involving the area of the surface and the volume between the surface and some fixed reference surface. It was first introduced by Xiao and the second named author in \cite{LX12}, where they used it to construct CMC hypersurfaces in the hyperbolic space $\H^{n+1}$ with prescribed asymptotic Dirichlet infinity (c.f. asymptotic Plateau problem in hyperbolic space). This paper builds up on our previous work (\cite{HLZ20}) on the {\mcf} ({\bf MCF} in short) in Fuchsian manifolds and \cite{HZZ19} in a class of warped product manifolds. As an application, we give an affirmative answer to Thurston's CMC foliation Conjecture \ref{Thurston-Conj} for a sub-class of {\afm}. More precisely we prove
  \begin{theorem}\label{main-thm-2}
Let $M^3$ be an {\afm} and $\Sigma =\Sigma(0)$ be the unique closed {\ms} with principal curvatures $\pm \lambda(x,0) \in (-1,1)$ in $M^3$. There exists a universal constant $\epsilon>0$ depending only on $M^3$ such that if the second fundamental form of $\Sigma$ satisfies 
$$
\|A_{\Sigma}\|_{C^1(\Sigma)} \leq \epsilon \,,
$$
then $M^3$ admits a unique global monotone smooth foliation by closed incompressible surfaces of constant mean curvature ranging from $-2$ to $2$.
\end{theorem}

As an immediate corollary of Theorem \ref{main-thm-2} we have
\bcor \label{cor-4}
Let $\Sigma$ be a closed oriented surface of genus at least two. There exists a universal neighborhood $U$ of the Fuchsian locus in {\af} space $\AF(\Sigma)$ such that every {\af} manifold in $U$ admits a unique global monotone smooth foliation by closed incompressible surfaces of constant mean curvature ranging from $-2$ to $2$.
\ecor
 
We remark that recently \cite{GLP21} utilized the mean convex MCF with surgery on {\qfm}s and proved a number of very interesting results: they showed any {\qfm} is foliated by mean convex surfaces (surfaces with {\mc} either all nonpositive or all nonnegative). Note that when $M$ is 
 {\af}, then Uhlenbeck (\cite{Uhl83}) has already shown that $M$ is foliated by (from one end to another) closed surfaces of negative {\mc} to the unique {\ms} to surfaces of positive {\mc}. There is a well-developed theory on the formation of singularities for mean convex MCFs in {\tm}s (see for example \cite{HS99, HS09, BH18, HK17, HK19} and others). However at this point how possible singularities can form for the MMCF is largely unknown. Our approach has a distinct feature: we do not rely on any surgeries to remove possible singularities may occur as the flow develops, instead, we look for appropriate initial geometrical conditions to entirely avoid any singularity. Therefore we will have to explore the special geometry of the ambient {\afm}. Our proof is of geometric flow in nature and is independent of previous results in the literature for the existence of CMC surfaces in hyperbolic manifolds.

\subsection{The setting}

We always assume the genus of any closed {\is} of any {\tm} is at least two so that it carries its own {\hym}. As we mentioned above, 
{\Fm}s are probably the most elementary complete, non-simply connected {\htm}s. A {\Fm} is obtained as a warped product and it contains a unique closed incompressible totally geodesic surface. An {\afm} $M^3$ is a class of quasi-Fuchsian manifold which admits a closed incompressible {\ms} $\Sigma = \Sigma(0)$ of {\pc}s less than one in magnitude. Let $\Sigma$ be this {\ms} (which is also the unique closed {\ms} in $M^3$ (\cite{Uhl83})). Then Uhlenbeck (\cite{Uhl83}) showed that the level sets 
$\{\Sigma(r)\}_{r\in \R}$ of $\Sigma$ form a foliation of $M^3$, where $r$ is the (signed) hyperbolic distance from $\Sigma$ (see Lemma \ref{U-metric}).
\begin{remark}\label{principle-curv-equidistant}
Assume the central minimal surface has principal curvatures $-1<-\lambda(p,0)\leqslant \lambda(p,0)<1$ at $p\in\Sigma$. Then the principal curvatures for the equidistant surface $\Sigma(r)$ at $(p, r)$ are given as (see e.g. \cite[Equation (3.10)]{Eps84}, \cite[Equation (2.2)]{GHW10}):
\begin{equation}
\tanh(\tanh^{-1}(-\lambda(p,0))+r) = \frac{\tanh(r)-\lambda(p,0)}{1-\lambda(p,0)\tanh(r)}  \quad \in (-1,1)
\end{equation}
and
\begin{equation}
\tanh(\tanh^{-1}(\lambda(p,0))+r) = \frac{\tanh(r)+\lambda(p,0)}{1+\lambda(p,0)\tanh(r)} \quad \in (-1,1)\,.
\end{equation}
\end{remark}
In the case when $M^3$ is Fuchsian, then the surface $\Sigma(r)$ is umbilic with constant {\pc} $\tanh(r)$. We are interested in the general case, so let us assume $M^3$ is \textit{not} Fuchsian, in this case, $\Sigma(r)$ is not {CMC}.

Our main analytical tool is the {\mmcf} (MMCF), which has the following form:
\be\label{mmcf}	
   \left\{
   \begin{aligned}
      \ppl{}{t}\,F(x,t)&=-(H(x,t)- c)\vnu(x,t)\ ,\\
      F(\cdot,0)&=F_{0}\ ,
   \end{aligned}
   \right.
\ene
where $H(x,t)$ is the {\mc} of the evolving surface $S(t)$ at $(x,t)$, with $S(0) =S_0$ and our convention of the {\mc} is the sum of the {\pc}s. We assume $c$ is a constant such that $c\in (-2,2)$. By symmetry we can only consider the case 
\begin{equation}\label{c-range}
    c\in [0,2).
\end{equation} Obviously when $c = 0$, this equation is the {\mcf} equation.

Recall that a smooth closed surface $S_0$ in $M^3$ is a (geodesic) graph over $\Sigma$ if there is a 
constant $c_0 > 0$ such that $\Theta = \langle{\vec{n}},{\vnu_0}\rangle \geq c_0$, where $\vec{n} = \ppl{}{r}$ is the unit {\nv} on $\Sigma$ and $\vnu_0$ is the unit normal vector on $S_{0}$. Note that $\Theta \in (0,1]$ if $S_0$ is a graph, and $\Theta \equiv 1$ if and only if $S_0$ is \emph{equidistant} from $\Sigma$ (sometimes called \emph{parallel} to $\Sigma$, or a \emph{level surface} to $\Sigma$). Again a quick note in the Fuchsian case which we showed in the earlier work \cite{HLZ20} that the {\es}s of the MCF of initial surface with 
$\Theta \equiv 1$ stay umbilic and converge to the unique {\tg} surface $\Sigma$. In fact, it is easy to check that the MMCF starting from an equidistant surface stays umbilic and converges to the unique surface of constant mean curvature $c$ in a Fuchsian manifold. Note that any {\afm} which is not Fuchsian, the unique {\ms} and its equidistant surfaces are not umbilic.

\subsection{Main Result}
As aforementioned, as a special case of the results proved by Mazzeo and Pacard \cite{MP11}, the ends of any {\qfm} are monotonically foliated by {\cmc} surfaces. This was reproved by Quinn \cite{Qui20} using the Epstein map construction. In \cite[Theorem 3.1]{CMS23}, Choudhury, Mazzoli and Seppi adapted Quinn's arguments and showed that for any {\qfm} $N \in \QF(\Sigma)$ there exists a neighborhood $U$ of $N$ in $\QF(\Sigma)$ and a constant $\epsilon_1 = \epsilon_1(N, U)>0$ such that the ends of every {\qfm} in $U$ are smoothly monotonically foliated by {\cmc} surfaces with mean curvature ranging in $(-2, -2+\epsilon_1) \cup (2-\epsilon_1, 2)$. In particular, for any almost Fuchsian manifold $M^3 \cong \Sigma \times \R$ with
\begin{equation}\label{dist-small}
\text{dist}_{\QF(\Sigma)}(M^3, \overline M)<\epsilon_2 \ll 1
\end{equation}
for some Fuchsian manifold $\overline M$ and a universal constant $\epsilon_2>0$, there exists a monotone foliation by CMC surfaces with mean curvature ranging in $(-2, -2+\epsilon_1) \cup (2-\epsilon_1, 2)$ in the two ends.
We remark that 
\begin{equation}\label{eps_1}
    \epsilon_1 = \epsilon_1 (M^3, \epsilon_2)>0
\end{equation}
here as in \cite[Theorem 3.1]{CMS23} does \textit{not} depend on the particular Fuchsian manifold $\overline M$. Then by a perturbation argument they showed the existence of a monotonic foliation by CMC surfaces with mean curvature ranging in $[-2+\epsilon_1, 2-\epsilon_1]$ in the compact region of this {\afm} $M^3$ provided
\begin{equation}\label{dist-small-1}
\text{dist}_{\QF(\Sigma)}(M^3, \overline M)<\epsilon_3(M^3, \overline{M}) \ll 1
\end{equation}
with $\epsilon_2 >0$ now depends also on the particular Fuchsian manifolds $\overline{M}$. This then yields a global monotonic CMC foliation for such an {\afm}.

Now let $M^3$ be an {\afm} and $\Sigma =\Sigma(0)$ be the unique closed {\ms} with principal curvatures $$\pm \lambda(x,0) \in (-1,1)$$ in $M^3$ and 
$$\lambda_{max} = \max_{x\in \Sigma} \lambda(x,0) \in [0,1).$$ In this paper, we prove that there exists a {\textit{universal}} constant $\epsilon>0$ depending only on $\epsilon_1>0$ in \eqref{eps_1} such that if the second fundamental form of $\Sigma$ satisfies 
\begin{equation}\label{small-A-1}
\|A_{\Sigma}\|_{C^1(\Sigma)}\leq \epsilon\,,
\end{equation}
then the {\mmcf} \eqref{mmcf} with $c\in [-2+\frac{\epsilon_1}{2}, 2-\frac{\epsilon_1}{2}]$ starting from an appropriate equidistant surface $\Sigma(r)$ will exist for all time, stays as geodesic graphs over $\Sigma$ and converges smoothly to a closed incompressible surface of constant mean curvature $c$. In particular, there exists a unique closed incompressible {\cmc} surface of any $c\in [-2+\frac{\epsilon_1}{2}, 2-\frac{\epsilon_1}{2}]$. We note that the condition \eqref{small-A-1} is considerably weaker than \eqref{dist-small}. More precisely, we prove
\bt\label{main}
Let $M^3$ be an {\afm} and $\Sigma =\Sigma(0)$ be the unique closed {\ms} with {\pc}s $\pm \lambda(x,0) \in (-1,1)$ in $M^3$. There exists a universal constant 
\be\label{unifconst}
    \epsilon \in \left[ 0, 1- \frac{1}{\left(1+\frac{\epsilon_1}{16}\right)^2}\right) \cap \left[0, \frac{\sqrt{\epsilon_1}}{12}\right]\cap \left[0, 7\times 10^{-6} \right],
\ene
where $\epsilon_1>0$ is from \eqref{eps_1}, such that if the {\sff} of $\Sigma$ satisfies 
$$
\|A_{\Sigma}\|_{C^1(\Sigma)} \leq \epsilon \,,
$$
then the MMCF \eqref{mmcf} with $c\in \left[0, 2-\frac{\epsilon_1}{2}\right]$ starting from any equidistant surface $S_0 = \Sigma(r) \subset M^3$ with 
\begin{equation}\label{r-range-1}
r \in \left[\tanh^{-1} \left(\frac{c}{2 - (2-\frac{c^2}{2})\lambda^2_{max}}\right), \frac12\cosh^{-1}\left(\frac18\ln \left(\frac{\epsilon_1}{\epsilon(16+\epsilon_1)}\right)\right)\right]
\end{equation}
satisfies:
\begin{enumerate}
\item
the flow exists for all time;
\item
the {\es}s $\{S(t)\}_{t \ge 0}$ stay smooth and remain as geodesic graphs over $\Sigma$ for all time;
\item
the {\es}s $\{S(t)\}_{t \ge 0}$ converge smoothly to a closed surface $S_{\infty}^c$ diffeomorphic to $\Sigma$ and is of {\cmc} $c$.
\end{enumerate}
\et
\begin{remark}
In fact, the initial surface $S_0$ in Theorem \ref{main} could be replaced any geodesic graph
with hyperbolic signed distance to $\Sigma$ within the range of \eqref{r-range-1} and
$$
\min_{S_0} \Theta \geq \sqrt{1-\epsilon}.
$$
Exploring the behavior of the modified mean curvature flow when $c$ falls within the range $(2 - \frac{\epsilon_1}{2}, 2)$ is an intriguing problem.
\end{remark}
As an immediate corollary of Theorem \ref{main} we have
\bcor\label{corollary-1}
Let $M^3$ be an {\afm} and $\Sigma =\Sigma(0)$ be the unique closed {\ms} with principal curvatures $\pm \lambda(x,0) \in (-1,1)$ in $M^3$. There exists a universal constant $\epsilon>0$ depending only on $M^3$ such that if the second fundamental form of $\Sigma$ satisfies 
$$
\|A_{\Sigma}\|_{C^1(\Sigma)} \leq \epsilon \,,
$$
then for any $c \in (-2,2)$, there exists a (unique) closed incompressible surface of {\cmc} $c$.
\ecor
\begin{proof}
The uniqueness follows from Theorem \ref{main-thm-2} above and the geometric maximal principle Theorem \ref{g-m-p}.
\end{proof}
As an application of the existence of closed incompressible surfaces of {\cmc}, we confirm Thurston's CMC foliation conjecture \ref{Thurston-Conj} for such a class of almost Fuchsian manifolds, see Theorem \ref{main-thm-2} and Corollary \ref{cor-4}.

\begin{remark}
It's important to note the primary distinction between our proof and the result presented in Corollary \ref{cor-4} when compared to the foliation result in \cite{CMS23}. Our proof is non-perturbative and applies to any {\af} manifold within a \textit{universal} neighborhood of the Fuchsian locus. This universality is independent of the specific Fuchsian manifold in consideration while the size of neighborhood in the main theorem of \cite{CMS23} depends on each point on the Fuchsian locus.  
\end{remark}

\subsection{Organization of the paper} 	
We recall some preliminary results in Section \S\ref{prelim}, including the explicit nature of the {\af} metric, the geometric {\maxp} and the evolution equation for the angle function $\Theta$ (see Lemma \ref{evoluation-eq-theta}). In Section \S\ref{angle-estimate-1} we meticulously compute each term in the evolution equation of the angle function $\Theta$. A key technical, yet essential, height estimate for the MMCF is derived in Section \S\ref{height-est} using the evolution equation for the height function. Our main results, as presented in Theorem \ref{main} and Theorem \ref{main-thm-2}, are proven in Section \S\ref{proof-of-main}.
\subsection{Acknowledgments} The first named author would like to thank Biao Wang and Andrea Seppi for many helpful discussions on the problem. The research of Z. Huang is partially supported by a PSC-CUNY grant. The research of L. Lin is partially supported by a UCSC grant. The research of Z. Zhang is partially supported by ARC Future Fellowship FT150100341.

\section{Preliminaries} \label{prelim} 
In this section, we fix our notations, and introduce some preliminary facts that will be used in this paper.

  

\subsection{Almost Fuchsian manifolds} 
 
Let $\Sigma$ be the unique {\ms} in the {\afm} $(M^3 = \Sigma \times \mathbb{R}, \,g_{\alpha\beta})$. The curvature tensor of $M^3$ is given by
\beq
   \widetilde R_{\alpha\beta\gamma\delta}=-
   (g_{\alpha\gamma}g_{\beta\delta}-
    g_{\alpha\delta}g_{\beta\gamma})\,.
\eeq
For the central {\ms} $\Sigma=\Sigma(0)$ and \textit{any} orthonormal frame $\{\vec{n}, \widetilde e_1, \widetilde e_2\}$ at $\forall (x,0)\in \Sigma$, we denote the second fundamental form of $\Sigma\subset M$ at $(x,0)$ by 
\begin{equation}\label{2ndff-A0}
    A_{\Sigma}= A_{\Sigma}(x, 0)= [h_{ij}]_{2\times{}2} = \begin{bmatrix}
a & b \\
b & -a
\end{bmatrix}
\end{equation}
and set $$\lambda^2(0) = a^2+b^2<1.$$
Here we use the convention such that $h_{ij}$ is given by, for $1\leq{}i,j\leq{}2$,
\begin{equation*}
   h_{ij}=\langle{\overline\nabla_{\widetilde e_{i}} \vec{n}}, {\widetilde e_{j}}\rangle
         =-\langle \vec{n}, \overline\nabla_{\widetilde e_i}\widetilde e_{j}\rangle\ ,
\end{equation*}
where $\overline\nabla$ is the Levi-Civita connection of $(M^3,g_{\alpha\beta})$. Note that $$\pm\lambda(0) = \pm\lambda(x,0) \in (-1,1)$$ are the principal curvatures of $\Sigma$ at $(x,0)$. We can construct a normal coordinate system in a collar neighborhood of $\Sigma$. More precisely, suppose $x=(x^{1},x^{2})$ is a coordinate system on $\Sigma$, and choose $\varepsilon>0$ to be sufficiently small, then the (local) 
diffeomorphism
\begin{align*}
   \Sigma \times(-\varepsilon,\varepsilon)&\to{}M\\
   (x^{1},x^{2},r)&\mapsto
   \exp_{x}(r\vec{n})
\end{align*}
induces a coordinate patch in $M^3$. Let $\Sigma(r)$ be the family of level sets with respect to $\Sigma$, i.e.
\beq
   \Sigma(r)=\{\exp_{x}(r\vec{n})\ |\ x\in{}\Sigma\}\ ,
   \quad{}r\in(-\varepsilon,\varepsilon)\ .
\eeq
The induced metric on $\Sigma(r)$ is denoted by $g(x,r) = g_{ij}(x,r)$, and the {\sff} is denoted by
$A(x,r)=[h_{ij}(x,r)]_{1\leq{}i,j\leq{}2}$. The {\mc} on $\Sigma(r)$ is then given by $H(x,r)=g^{ij}(x,r)h_{ij}(x,r)$.

The curvature tensor $\widetilde R_{\alpha\beta\gamma\delta}$ of $(M^3, g_{\alpha\beta})$ has six components, which are 
not completely independent because of Bianchi identities. In the collar neighborhood of $\Sigma$, these components can be 
classified into three groups:
\begin{enumerate}
   \item there are three curvature equations of the form
         $\widetilde R_{i3j3}=-g_{ij}$, here $1\leq{}i,j\leq{}2$,
   \item two of remaining curvature equations have the form
         $\widetilde R_{ijk3}=0$,
         which are called the Codazzi equations, and
   \item the Gauss equation
         $\widetilde R_{1212}=-g_{11}g_{22}+g_{12}^{2}$.
\end{enumerate}
This enables us to solve for the metric $g$ on $M^3$ explicitly, and the solution is due to Uhlenbeck \cite{Uhl83}. In the following we will denote $\vec{n} = \frac{\partial}{\partial r}$.
\bl\label{U-metric}
 With respect to the dual basis $$\{dr=\vec{n}^*, \widetilde e_1^*, \widetilde e_2^*\}\,,$$ 
 the metric $g$ in an almost Fuchsian manifold $M^3 = \Sigma \times \mathbb{R}$ has the form:
\begin{equation}\label{metric}
   g=dr^2+g(r)\ ,
\end{equation}
where $r \in (-\infty,\infty)$ and
\allowdisplaybreaks
\begin{align}\label{u-metric-2}
    g(r)= \begin{bmatrix} \widetilde  e_1^* & \widetilde e_2^*\end{bmatrix}
    E^2 \begin{bmatrix} \widetilde  e_1^*\\ \widetilde e_2^* \end{bmatrix} = \begin{bmatrix} \widetilde  e_1^* & \widetilde e_2^* \end{bmatrix} \left[\cosh(r)I_2+\sinh(r)A_{\Sigma}\right]^2\begin{bmatrix} \widetilde  e_1^*\\ \widetilde e_2^*\end{bmatrix} 
\end{align}
for the symmetric matrix 
\allowdisplaybreaks
\begin{align}\label{matrix-E}
E & =\cosh(r)I_2+\sinh(r)A_{\Sigma}\notag\\
&= \begin{bmatrix}
\cosh(r)+a\sinh(r) & b\sinh(r) \\
b\sinh(r) & \cosh(r)-a\sinh(r)
\end{bmatrix}\,.
\end{align}
\el
\begin{remark} \label{exp-ex-1}
Uhlenbeck in \cite{Uhl83} used the isothermal coordinates on $\Sigma$ to construct the explicit metric on $M^3$ while we use the orthonormal coordinates (see the subsection \S 3.1). Explicitly we have
 \begin{align*}
E^2 = 
    \begin{bmatrix}
1+(1+\lambda^2(0))\sinh^2(r) + a\sinh(2r) & b\sinh(2r) \\
b \sinh(2r) & 1+(1+\lambda^2(0))\sinh^2(r) - a\sinh(2r)
\end{bmatrix}.
\end{align*}   
We also have
\allowdisplaybreaks
\begin{equation}\label{Erderivative}
\frac{\p E}{\p r}=\begin{bmatrix}
\sinh(r)+a\cosh(r) & b\cosh(r) \\
b\cosh(r) & \sinh(r)-a\cosh(r)
\end{bmatrix}\,,
\end{equation}
\begin{equation}\label{Etimes-er}
E\frac{\p E}{\p r}=\begin{bmatrix}
\frac{1+\lambda^2(0)}{2}\sinh(2r)+a\cosh(2r) & b\cosh(2r) \\
b\cosh(2r) & \frac{1+\lambda^2(0)}{2}\sinh(2r)-a\cosh(2r)
\end{bmatrix}\,,
\end{equation}
\begin{align}\label{Edeterminant}
\det(E) = \cosh^2(r)-\sinh^2(r)\det(A_{\Sigma}) &= \cosh^2(r)-\sinh^2(r)\lambda^2(0)\notag\\
&= 1+(1-\lambda^2(0))\sinh^2(r)\,,
\end{align} 
\begin{align}\label{Einverse}
E^{-1}&=\frac{1}{\det(E)}\begin{bmatrix}
\cosh(r)-a\sinh(r) & -b\sinh(r) \\
-b\sinh(r) & \cosh(r)+a\sinh(r)
\end{bmatrix}\\
&=\frac{1}{1+(1-\lambda^2(0))\sinh^2(r)}\begin{bmatrix}
\cosh(r)-a\sinh(r) & -b\sinh(r) \\
-b\sinh(r) & \cosh(r)+a\sinh(r)
\end{bmatrix}\,.\notag
\end{align}
and 
\allowdisplaybreaks
\begin{align}\label{ginverse}
g^{-1}&=E^{-2}\\
=&\left(1+(1-\lambda^2(0))\sinh^2(r)\right)^{-2}\cdot\notag\\
&\begin{bmatrix}
1+(1+\lambda^2(0))\sinh^2(r)-a\sinh(2r) & -b\sinh(2r) \\
-b\sinh(2r) &1+(1+\lambda^2(0))\sinh^2(r)+a\sinh(2r)
\end{bmatrix}\,.\notag
\end{align}
\end{remark}
\subsection{Geometric maximum principle} 
We state a geometric version of the Hopf's maximum principle that is also known as the \textit{geometric maximum principle} or \textit{geometric comparison principle}. 
\begin{theorem}\label{g-m-p}(c.f. \cite[Theorem 3.1.5]{L2013}) Let $\Sigma_1$ and $\Sigma_2$ be two surfaces in $M^3$ with a common tangent point $p$. If $\Sigma_1$ lies above $\Sigma_2$ in the direction $\vec{n}=\frac{\partial}{\partial r}$ around $p$, then the mean curvatures of $\Sigma_1$ and $\Sigma_2$ at $p$, denoted by $H_1(p)$ and $H_2(p)$ respectively, have the following comparison:
\begin{equation}
    H_1(p) \leq H_2(p).
\end{equation}
\end{theorem}
Note that the strict inequality is not true in general. 

\subsection{Evolution equations for modified mean curvature flows} 
The short-time existence of the MMCF \eqref{mmcf} is standard for closed immersions, see e.g. \cite{HP96}. In general the flow with initial compact surface develops singularities in finite time along the modified mean curvature flow in {\eus}, c.f. \cite {Hui84, Hui86}.

Our estimates will also involve the {\it height function} $u(x,t)$ and the {\it angle function (or the gradient)}
$\Theta(x,t)$ on {\es}s $S(t)$:
\begin{align}
   u(x,t)&=\ell(F(x,t))\\
      \label{eq:gradient function}
   \Theta(x,t)&=\langle{\vnu(x,t)},{\vec{n}}\rangle\ .
\end{align}
Here $\ell(p) = \pm dist(p,\Sigma)$ for all $p \in M^3$, the distance to the fixed reference surface $\Sigma$, and $\vnu$ is the unit normal vector field on $S(t)$. We always have $\Theta (x,t) \in [0,1]$. It is clear that the surface $S(t)$ is a graph over $\Sigma$ if
$\Theta > 0$ on $S(t)$. In this subsection, we derive the {\ee} of the angle function $\Theta$ on the evolving surface
$S(t)$, $t\in[0,T)$. Note that by our convention, the geometric quantities and operators involved are:
\begin{enumerate}
   \item
   the induced metric of $S(t)$: $g(t)= [g_{ij}(t)]$;
   \item
    the {\sff} of $S(t)$: $A(\cdot,t)=[h_{ij}(\cdot,t)]$;
   \item
   the {\mc} of $S(t)$ with respect to the normal vector pointing to $\Sigma$: $H(\cdot,t)=g^{ij}h_{ij}$;
   \item
   the square norm of the {\sff} of $S(t)$:
         \begin{equation*}
            |A(\cdot,t)|^{2}=g^{ij}g^{kl}h_{ik}h_{jl}\ ;
         \end{equation*}
   \item
   the covariant derivative of $S(t)$ is denoted by $\nabla$ and the Laplacian on $S(t)$ is given by
         $\Delta=g^{ij}\nabla_{i}\nabla_{j}$.
\end{enumerate}

\begin{lem}\label{evoluation-eq-theta}
 The evolution equation of the angle function $\Theta = \langle\vec{\nu}, \vec{n}\rangle$ is:
\begin{align}\label{eq:theta-evolution-new}
\left(\frac{\p}{\p t}-\Delta\right)\Theta =& (|A|^2-2)\Theta+\frac{1}{2}(\overline\nabla_{\vec{\nu}} L_{\vec{n}} g)(e_i, e_i)
-(\overline\nabla_{e_i}L_{\vec{n}} g)(\vec{\nu}, e_i)\notag\\ 
&-(L_{\vec{n}} g)(e_i, e_j)A_{ij} +  \frac{c}{2} L_{\vec{n}} g(\vec{\nu}, \vec{\nu}),
\end{align}
where $\{\vec{\nu}, e_1, e_2\}$ is an local orthonormal frame on the evolving surface $S(t)$ and $L_{\vec{n}} g$ denotes the Lie derivative of $g$ over $M$ in the direction of $\vec{n}$. 
\end{lem}

\begin{proof}
The evolution equation of the angle function under the mean curvature flow in a general production manifold is given in our previous work \cite[Theorem 2.3]{HLZ20}. There is only one extra term in the evolution equation of $\Theta$ under the MMCF, which comes from the $\frac{\partial \Theta}{\partial t}$ term. We use the coordinate system 
set-up in \cite{Hui86}, i.e. a normal coordinate system 
$\{y_\alpha\}$ for $F(p, t)$ in $M$ with the frame vector 
for the first coordinate is $-\vec{\nu}$ at time $t$.

Let $\vec{\nu}=\vec{\nu}^\alpha\frac{\p}{\p y^\alpha}$ and $\vec{n}=\vec{n}^
\alpha\frac{\p}{\p y^\alpha}$. We have $\frac{\p \Theta}
{\p t}=\frac{\p (g^{\alpha\beta}\vec{\nu}^\alpha \vec{n}^\beta)}{\p 
t}$. There are three terms to consider. 
\begin{equation}
\begin{split}  
\frac{\p g_{\alpha\beta}}{\p t}
&= \frac{\p}{\p t}\left\langle\frac{\p}{\p y^\alpha}, \frac{\p}
{\p y^{\beta}}\right\rangle \\
&= -(H-c)\vec{\nu}\left\langle\frac{\p}{\p y^\alpha}, \frac{\p}{\p y^
{\beta}}\right\rangle = 0, \nonumber
\end{split}
\end{equation}
because the Christoffel symbols vanish at the point. Define $\frac{\p \vec{\nu}}{\p t}$ to be $\frac{\p \vec{\nu}
^\alpha}{\p t}\frac{\p}{\p y^\alpha}$ and  $\frac
{\p \vec{n}}{\p t}$ to be $\frac{\p \vec{n}^\alpha}{\p t}\frac
{\p}{\p y^\alpha}$. We can see 
$$\frac{\p \Theta}{\p t}=\<\frac{\p \vec{\nu}}{\p t}, \vec{n}\>+
\<\vec{\nu}, \frac{\p\vec{n}}{\p t}\>.$$
Since $\frac{\p g_{\alpha\beta}}{\p t}=0$ and $\<
\vec{\nu}, \vec{\nu}\>=1$, one easily see $\<\frac{\p \vec{\nu}}{\p t}, 
\vec{\nu}\>=0$, and so $\frac{\p \vec{\nu}}{\p t}$ is a tangent vector for the evolving surface. One can then choose a coordinate for the reference surface around $p$, $\{x^i\}$. We also have 
$$\frac{\p g_{\alpha\beta}}{\p x^k}=\frac{\p F}{\p x^k}
\left\langle \frac{\p}{\p y^\alpha}, \frac{\p}{\p y^{\beta}}\right\rangle=0$$
by the choice of $\{y^\alpha\}$. Now one computes $
\frac{\p \vec{\nu}}{\p t}$ as follows.
\begin{equation}
\begin{split}
\frac{\p \vec{\nu}}{\p t}
&= g^{kl}\<\frac{\p \vec{\nu}}{\p t}, \frac{\p F}{\p x^k}\>\frac
{\p F}{\p x^l} \\
&= -g^{kl}\left\langle \vec{\nu}, \frac{\p^2 F}{\p t\p x^k}\right\rangle\frac{\p F}
{\p x^l} \\
&= -g^{kl}\left\langle \vec{\nu}, \frac{\p (-(H-c)\vec{\nu}^\alpha)}{\p x^k}\frac
{\p}{\p y^\alpha}\right\rangle \frac{\p F}
{\p x^l} \\
&= g^{kl}\frac{\p H}{\p x^k}\frac{\p F}{\p x^l} \\
&= \nabla H, \nonumber
\end{split}
\end{equation}
where $\frac{\p F}{\p x^k}=\frac{\p F^\alpha}{\p x^k}
\frac{\p}{\p y^\alpha}$, $\frac{\p F}{\p t}=\frac{\p F^
\alpha}{\p t}\frac{\p}{\p y^\alpha}=-(H-c)\vec{\nu}^\alpha\frac
{\p}{\p y^\alpha}$ and $\frac{\p^2 F}{\p t\p x^k}
=\frac{\p^2 F^\alpha}{\p t\p x^k}\frac{\p}{\p y^\alpha}=
\frac{\p (-(H-c)v^\alpha)}{\p x^k}\frac{\p}{\p y^\alpha}$. 
The second equality in the above computation is true 
because $\<v, \frac{\p F}{\p x^k}\>=0$ and $\frac{\p g_
{\alpha\beta}}{\p t}=0$. For the fourth equality, one 
makes use of $\<\vec{\nu}, \vec{\nu}\>=1$ and $\frac{\p g_{\alpha
\beta}}{\p x^k}=0$. We can treat the other term similarly.
$$\frac{\p \vec{n}}{\p t}=\frac{\p \vec{n}^\alpha}{\p t}\frac
{\p}{\p y^\alpha}=-(H-c)\vec{\nu}(\vec{n}^\alpha)\frac{\p}
{\p y^\alpha}=-H\overline\nabla_{\vec{\nu}} \vec{n} + c\overline\nabla_{\vec{\nu}} \vec{n},$$
where the last one is again by the choice of $\{y^
\alpha\}$. Thus the extra term in the evolution equation of $\Theta$ under the MMCF is
\begin{align}
c\langle \vec{\nu}, \overline\nabla_{\vec{\nu}} \vec{n} \rangle 
 = & \frac{c}{2} L_{\vec{n}} g(\vec{\nu}, \vec{\nu}).
\end{align}
\end{proof}

\begin{lem}\label{u-t-eq}
\be\label{ut}
u_t = -(H-c)\Theta.
\ene
\end{lem}
\begin{proof}
It follows directly from definition.
\end{proof}

\section{Estimates on the angle function}\label{angle-estimate-1}
Our approach of the proof is quite straightforward: the maximum principle, together with the evolution equation of $\Theta=\<\vec{\nu}, \vec{n}\>\geqslant 0$ in Lemma \ref{evoluation-eq-theta}. We aim to prove that the {\es}s stay graphical under appropriate initial conditions (Lemma \ref{keylemma}). The uniform positive lower bound of the angle function $\Theta = \langle \vec{n}, \vec{\nu} \rangle$ gives the uniform $C^1$-estimate of the graph function which represents the evolving surface. Therefore once we have established uniform positive lower bound for $\Theta$, standard parabolic theory (\cite {LSU68}) gives bounds for all higher derivatives. In particular the {\sff} for the {\es} $S(t)$ in $M^3$ is uniformly bounded and the {\mmcf} exists for all time.

In order to execute this approach, we only need to consider the points on the evolving surface $S(t)$ where $\Theta$ achieves the minimum. Without loss of generality, we assume that the minimum point of $\Theta$ on $S(t)$ is $(p,r_0)\in \Sigma(r_0)$ and $\Theta(p,r_0)<1$.

\subsection{The local coordinates} 
Now fix this point of interest $(p,r_0) \in \Sigma(r_0)$ on the evolving surface, at $(p,0)$ on $\Sigma=\Sigma(0)$ we choose the local orthonormal frame
$$
 \{\widetilde {e}_1, \widetilde{e}_2\}(p,0) $$
at $(p,0)$ to match $\{\widehat{e}_1, \widehat{e}_2\}(p,r_0)$. Then we can use $\Sigma$-parallel construction to get a local orthonormal frame $\{\vec{n}, \widetilde e_1, \widetilde e_2\}_{r=0}$ in a neighborhood of $(p,0)$ with 
\begin{equation}\label{frameconst1}
\nabla^{\Sigma} {\widetilde e_1} = \nabla^{\Sigma}\widetilde e_2 =0
\end{equation}
at $(p,0)$ (and so $[\widetilde e_1, \widetilde e_2]=0$ at $(p,0)$). Then we can use the $\vec{n}$-parallel construction of $\{\vec{n}, \widetilde e_1, \widetilde e_2\}_{|_{r=0}}$ to get a frame over $M$, in other words, $$[\vec{n}, \widetilde e_i] =0, \quad i=1,2.$$ 
Note that since $\vec{n}$ is not a Killing vector field, the $\vec{n}$-parallel constructed local frame over $M$ is no longer orthonormal away from the central surface.  We let
\allowdisplaybreaks
\begin{align}\label{dual-e}
\begin{bmatrix}
\widehat e_1 \\
\widehat e_2
\end{bmatrix} (\cdot, r)  = E^{-1}\begin{bmatrix}
\widetilde e_1 \\
\widetilde e_2
\end{bmatrix} (\cdot, r) \quad \text{and}\quad
\begin{bmatrix}
\widehat e_1^* \\
\widehat e_2^*
\end{bmatrix}(\cdot, r) = E \begin{bmatrix}
\widetilde e_1^* \\
\widetilde e_2^*
\end{bmatrix}(\cdot, r) ,
\end{align}
where $E$ is the symmetric matrix 
$E=\cosh(r)I_2+\sinh(r)A_{\Sigma}$
as in \eqref{matrix-E}. 

At the minimum point $(p,r_0) \in \Sigma(r_0)$ of $\Theta$, since $\Theta(p,r_0)<1$ then $\vec{\nu}$ and $\vec{n}$ spans a $2$-plane. Then at this point of interest we have two orthonormal frames $\{\vec{\nu}, e_1, e_2\}$ and $\{\vec{n}, \widehat e_1, \widehat e_2\}$ where $\vec{\nu}$ is the unit normal vector of the evolving surface and we choose the other two vectors $\{e_1, e_2\}$ and $\{\widehat e_1, \widehat e_2\}$ to be symmetric (after an appropriate rotation) with respect to the $\vec{n}$-$\vec{\nu}$ plane, namely, around this point of interest $(p,r_0) \in \Sigma(r_0)$, we have  

\allowdisplaybreaks
\begin{equation}\label{widehat-n}
    \left\{
\begin{aligned}
\vec{n}&=\Theta \vec{\nu}+\frac{\sqrt{1-\Theta^2}}{\sqrt{2}} e_1+\frac{\sqrt{1-\Theta^2}}{\sqrt{2}}e_2,\\
\widehat e_1&=\frac{-\sqrt{1-\Theta^2}}{\sqrt{2}}\vec{\nu}+\frac{\Theta+1}{2}e_1+\frac{\Theta-1}{2} e_2,\\
\widehat e_2&=\frac{-\sqrt{1-\Theta^2}}{\sqrt{2}}\vec{\nu}+\frac{\Theta-1}{2}e_1+\frac{\Theta+1}{2}e_2.
\end{aligned}
\right.
\end{equation}
and the inverse
\allowdisplaybreaks
\begin{equation}\label{coordinate-change-1}
    \left\{
\begin{aligned}
\vec{\nu} &=\Theta \vec{n}+\frac{-\sqrt{1-\Theta^2}}{\sqrt{2}}\widehat e_1+\frac{-\sqrt{1-\Theta^2}}{\sqrt{2}}\widehat e_2,\\
e_1 &=\frac{\sqrt{1-\Theta^2}}{\sqrt{2}}\vec{n}+\frac{\Theta+1}{2}\widehat e_1+\frac{\Theta-1}{2}\widehat e_2,\\
e_2&=\frac{\sqrt{1-\Theta^2}}{\sqrt{2}}\vec{n}+\frac{\Theta-1}{2}\widehat e_1+\frac{\Theta+1}{2}\widehat e_2.
\end{aligned}
\right.
\end{equation}

\begin{remark}
Note that the $\{ \widehat e_1, \widehat e_2\}$ defined in \eqref{dual-e} might not be symmetric at points other than $p$.
\end{remark}
Note also that since 
$$E^{-1}\begin{bmatrix}
\widetilde e_1 \\
\widetilde e_2
\end{bmatrix}\odot (\widetilde e_1, \widetilde e_2)E^{-1}= \begin{bmatrix}
\widetilde e_1 \\
\widetilde e_2
\end{bmatrix}_{|_{r=0}}\odot \begin{bmatrix}
\widetilde e_1 &
\widetilde e_2
\end{bmatrix}_{|_{r=0}} = I_2,$$
we have
$$\begin{bmatrix}
\widetilde e_1 \\
\widetilde e_2
\end{bmatrix}\odot \begin{bmatrix}
\widetilde e_1 &
\widetilde e_2
\end{bmatrix} =E^2,$$
where $\odot$ is just the metric product on $M$.

Now by the general formula for Levi-Civita connection $\nabla$ for the metric $\langle\cdot, \cdot\rangle$, 
$$2\langle\nabla_X Y, Z\rangle=X\langle Y, Z\rangle+Y\langle Z, X\rangle-Z\langle X, Y\rangle+\langle [X, Y], Z\rangle-\langle [Y, Z], X\rangle-\langle [X, Z], Y\rangle,$$
and the fact that Lie brackets between $\vec{n}$, $\widetilde e_1$ and $\widetilde e_2$ all vanish, we have
\allowdisplaybreaks
\begin{equation}\label{nabla-n-n}
\overline\nabla_{\vec{n}} \vec{n}=0,
\end{equation}
$$2\<\overline \nabla _{\widetilde e_1} \vec{n}, \widetilde e_1\>=\frac{\p g_{11}}{\p r}, ~~2\<\overline \nabla _{\widetilde e_1} \vec{n}, \widetilde e_2\>=\frac{\p g_{12}}{\p r},$$
$$2\<\overline \nabla _{\widetilde e_2} \vec{n}, \widetilde e_1\>=\frac{\p g_{12}}{\p r}, ~~2\<\overline \nabla _{\widetilde e_2} \vec{n}, \widetilde e_2\>=\frac{\p g_{22}}{\p r}.$$
Let's set 
\allowdisplaybreaks
\begin{equation}\label{nabla_e_n}
    \begin{bmatrix}
\overline\nabla_{\widetilde e_1} \vec{n} \\
\overline\nabla_{\widetilde e_2} \vec{n}
\end{bmatrix}=\begin{bmatrix}
\alpha & \beta \\
\delta & \eta
\end{bmatrix}\begin{bmatrix}
\widetilde e_1 \\
\widetilde e_2
\end{bmatrix}
\end{equation}
with the $\vec{n}$-components clearly vanishing. So we have
$$
2\begin{bmatrix}
g_{11} & g_{12} \\
g_{12} & g_{22}
\end{bmatrix}\cdot \begin{bmatrix}
\alpha & \delta \\
\beta & \eta
\end{bmatrix}=2\begin{bmatrix}
\widetilde e_1 \\
\widetilde e_2
\end{bmatrix}\odot (\widetilde e_1, \widetilde e_2)\begin{bmatrix}
\alpha & \delta \\
\beta & \eta
\end{bmatrix}=\frac{\p}{\p r}\begin{bmatrix}
g_{11} & g_{12} \\
g_{12} & g_{22}
\end{bmatrix}.
$$

\begin{remark}\label{remark1}
    Since $(g_{ij})=E^2$ is the metric on the equidistant surface $\Sigma(r)$ and only the matrices $I_2$ and $A_{\Sigma}=\begin{bmatrix}
a & b \\
b & -a
\end{bmatrix}$ are involved in the calculation, such matrix multiplications with only $\frac{\p}{\p r}$ involved are commutative. For example, $E$ commutes with $\frac{\partial E}{\partial r}$ and we know $\beta=\delta$, i.e. the above matrix is symmetric and written as
$$F=\begin{bmatrix}
\alpha & \beta \\
\beta & \eta
\end{bmatrix}=E^{-1}\frac{\p E}{\p r}\,,$$
which commutes with $E$ and $\frac{\p E}{\p r}$.
\end{remark}
By the above remark, we have $\delta =\beta$ and (using \eqref{Erderivative} and \eqref{Einverse})
\allowdisplaybreaks
\begin{align}\label{matrixF}
F &= \begin{bmatrix}
\alpha & \beta \\
\beta & \eta
\end{bmatrix}=E^{-1}\frac{\p E}{\p r}\\
&= 
\frac{1}{1+(1-\lambda^2(0))\sinh^2(r)}\begin{bmatrix}
\frac{1}{2}\sinh(2r)(1-\lambda^2(0))+a & b \\
b & \frac{1}{2}\sinh(2r)(1-\lambda^2(0))-a
\end{bmatrix}\,.\notag
\end{align}

Now using \eqref{dual-e}, \eqref{nabla_e_n} (with $\delta =\beta$) and \eqref{matrixF}, we have
$$\begin{bmatrix}
\overline\nabla_{\widehat e_1} \vec{n} \\
\overline\nabla_{\widehat e_2} \vec{n}
\end{bmatrix}=E^{-1}\begin{bmatrix}
\overline\nabla_{\widetilde e_1} \vec{n} \\
\overline\nabla_{\widetilde e_2} \vec{n}
\end{bmatrix}=E^{-1}\begin{bmatrix}
\alpha & \beta \\
\beta & \eta
\end{bmatrix}\begin{bmatrix}
\widetilde e_1 \\
\widetilde e_2
\end{bmatrix}=E^{-1}\begin{bmatrix}
\alpha & \beta \\
\beta & \eta
\end{bmatrix}E\begin{bmatrix}
\widehat e_1 \\
\widehat e_2
\end{bmatrix}=\begin{bmatrix}
\alpha & \beta \\
\beta & \eta
\end{bmatrix}\begin{bmatrix}
\widehat e_1 \\
\widehat e_2
\end{bmatrix}$$
by the commutativity mentioned above (see Remark \ref{remark1}).

\subsection{Relating the second fundamental forms of evolving surface and equidistant surface} At the local minimum point $(p,r_0)\in \Sigma(r_0)$ of $\Theta$, we have $\nabla\Theta=0$. Thus
\allowdisplaybreaks
\begin{align*}
0
=& e_1\<\vec{\nu}, \vec{n}\>(p,r_0) \\
=& \<\overline\nabla_{e_1} \vec{\nu}, \vec{n}\>+\<\vec{\nu}, \overline\nabla_{e_1}\vec{n}\> \\
=& \<\overline\nabla_{e_1} \vec{\nu}, \Theta \vec{\nu}+\frac{\sqrt{1-\Theta^2}}{\sqrt{2}} e_1+\frac{\sqrt{1-\Theta^2}}{\sqrt{2}}e_2\> \\
&+\<\Theta \vec{n}+\frac{-\sqrt{1-\Theta^2}}{\sqrt{2}}\widehat e_1+\frac{-\sqrt{1-\Theta^2}}{\sqrt{2}}\widehat e_2, \overline\nabla_{\frac{\sqrt{1-\Theta^2}}{\sqrt{2}}n+\frac{\Theta+1}{2}\widehat e_1+\frac{\Theta-1}{2}\widehat e_2}\vec{n}\> \\
=& \frac{\sqrt{1-\Theta^2}}{\sqrt{2}}(A_{11}+A_{12})-\frac{\sqrt{1-\Theta^2}}{\sqrt{2}}\<\widehat e_1+\widehat e_2, \frac{\Theta+1}{2}\overline\nabla_{\widehat e_1} \vec{n}+\frac{\Theta-1}{2}\overline\nabla_{\widehat e_2} \vec{n} \> \\
=& \frac{\sqrt{1-\Theta^2}}{\sqrt{2}}(A_{11}+A_{12})-\frac{\sqrt{1-\Theta^2}}{\sqrt{2}}\<\widehat e_1+\widehat e_2, \frac{\Theta+1}{2}(\alpha\widehat e_1+\beta\widehat e_2)+\frac{\Theta-1}{2}(\beta\widehat e_1+\eta\widehat e_2) \> \\
=& \frac{\sqrt{1-\Theta^2}}{\sqrt{2}}(A_{11}+A_{12})-\frac{\sqrt{1-\Theta^2}}{\sqrt{2}}\(\frac{\Theta+1}{2}\alpha+\Theta \beta+\frac{\Theta-1}{2}\eta\), 
\end{align*}
where $A_{ij} = A_{ij}(r_0)$ denotes the second fundamental form of the evolving surface. Similarly, we have
\allowdisplaybreaks
\begin{align*}
0
=& e_2\<\vec{\nu}, \vec{n}\> \\
=& \<\overline\nabla_{e_2} \vec{\nu}, \vec{n}\>+\<\vec{\nu}, \overline\nabla_{e_2}\vec{n}\> \\
=& \<\overline\nabla_{e_2} \vec{\nu}, \Theta \vec{\nu}+\frac{\sqrt{1-\Theta^2}}{\sqrt{2}} e_1+\frac{\sqrt{1-\Theta^2}}{\sqrt{2}}e_2\> \\
&+\<\Theta \vec{n}+\frac{-\sqrt{1-\Theta^2}}{\sqrt{2}}\widehat e_1+\frac{-\sqrt{1-\Theta^2}}{\sqrt{2}}\widehat e_2, \overline\nabla_{\frac{\sqrt{1-\Theta^2}}{\sqrt{2}}n+\frac{\Theta-1}{2}\widehat e_1+\frac{\Theta+1}{2}\widehat e_2}\vec{n}\> \\
=& \frac{\sqrt{1-\Theta^2}}{\sqrt{2}}(A_{12}+A_{22})-\frac{\sqrt{1-\Theta^2}}{\sqrt{2}}\<\widehat e_1+\widehat e_2, \frac{\Theta-1}{2}\overline\nabla_{\widehat e_1} \vec{n}+\frac{\Theta+1}{2}\overline\nabla_{\widehat e_2} \vec{n} \> \\
=& \frac{\sqrt{1-\Theta^2}}{\sqrt{2}}(A_{12}+A_{22})-\frac{\sqrt{1-\Theta^2}}{\sqrt{2}}\<\widehat e_1+\widehat e_2, \frac{\Theta-1}{2}(\alpha\widehat e_1+\beta\widehat e_2)+\frac{\Theta+1}{2}(\beta\widehat e_1+\eta\widehat e_2) \> \\
=& \frac{\sqrt{1-\Theta^2}}{\sqrt{2}}(A_{12}+A_{22})-\frac{\sqrt{1-\Theta^2}}{\sqrt{2}}\(\frac{\Theta-1}{2}\alpha+\Theta \beta+\frac{\Theta+1}{2}\eta\). 
\end{align*}
So at the minimum point $(p,r_0)$ of $\Theta$ we have
$$(A_{11}+A_{12})-\(\frac{\Theta+1}{2}\alpha+\Theta \beta+\frac{\Theta-1}{2}\eta\)=0,$$
$$(A_{12}+A_{22})-\(\frac{\Theta-1}{2}\alpha+\Theta \beta+\frac{\Theta+1}{2}\eta\)=0,$$
which are equivalent to 
\begin{equation}
\label{eq:gradient=0-lin}
A_{11}+2A_{12}+A_{22}=\Theta(\alpha+\eta+2\beta),  ~~~~A_{11}-A_{22}=\alpha-\eta.
\end{equation}
In fact, we have
\begin{align*}
A_{11} &= \frac{H+\alpha-\eta}{2},\\
A_{22} &= \frac{H-\alpha+\eta}{2},\\
A_{12} &= \frac{\Theta(\alpha+\eta+2\beta) - H}{2},
\end{align*}
where $H = A_{11}+A_{22}$ is the mean curvature at the point of interest of the evolving surface. Then
\begin{equation}\label{eq:gradient=0-lin-2}
    A = \begin{bmatrix}
        A_{11} & A_{12}\\
        A_{21} & A_{22}
    \end{bmatrix} = \begin{bmatrix}
        \frac{H+\alpha-\eta}{2} & \frac{\Theta(\alpha+\eta+2\beta) - H}{2}\\
        \frac{\Theta(\alpha+\eta+2\beta) - H}{2} & \frac{H-\alpha+\eta}{2}
    \end{bmatrix}.
\end{equation}
Note that the {\sff} of the equidistant surface $\Sigma(r)$ is (c.f. \eqref{2ndff-A0} and \eqref{matrixF} when $r=0$)
\begin{equation}\label{2ndff-1}
  A(r):= \frac{1}{2}L_{\vec{n}} g=(\widetilde  e_1^*, \widetilde e_2^*)E\frac{\p E}{\p r}(\widetilde  e_1^*, \widetilde e_2^*)^T.
\end{equation}
Thus \eqref{eq:gradient=0-lin} establishes the relation between the second fundamental forms of the evolving surface and the equidistant surface at the minimum point $(p,r_0)$ of $\Theta$.
\begin{remark} \label{mean-curvature-99}
By Remark \ref{principle-curv-equidistant}, we have at $(p,r_0)$
\begin{align}\label{alpha-plus-eta}
    \alpha+\eta &=\tanh(\tanh^{-1}(-\lambda(p,0))+r_0) + \tanh(\tanh^{-1}(\lambda(p,0))+r_0) \notag\\
    &= \frac{\sinh(2r_0)(1-\lambda^2(p,0))}{1+(1-\lambda^2(p,0))\sinh^2(r_0)} = \frac{2(1-\lambda^2(p,0)) \tanh(r_0)}{1- \lambda^2(p,0) \tanh^2(r_0)}\,,
    \end{align}
    and
\begin{align}\label{alpha-minus-eta}
    \alpha-\eta
    = \frac{2a}{1+(1-\lambda^2(p,0))\sinh^2(r_0)}\,,
    \end{align}    
    where we used \eqref{matrixF}.
\end{remark}

\subsection{The term $L_{\vec{n}} g(e_i, e_j)A_{ij}$}
Now we calculate the terms appearing in the evolution equation (\ref{eq:theta-evolution-new}) of $\Theta$ explicitly for estimation. We start with the term $(L_{\vec{n}} g)(e_i. e_j)A_{ij}$. Recall that
\begin{align*}
g &=dr^2+(\widetilde  e_1^*, \widetilde e_2^*)E^2(\widetilde  e_1^*, \widetilde e_2^*)^T\\
&=dr^2+(\widetilde  e_1^*, \widetilde e_2^*)[\cosh(r)I_2+\sinh(r)A_{\Sigma}]^2(\widetilde  e_1^*, \widetilde e_2^*)^T,
\end{align*}
we have 
\begin{align}\label{L-n-g-1}
    L_{\vec{n}} g
    &= 2(\vec{n}^*, \widetilde  e_1^*, \widetilde e_2^*) 
    \begin{bNiceArray}{ccc}
0       & 0 & 0      \\
0       & \Block{2-2} {E\frac{\p E}{\p r}} \\
0       &   &   
    \end{bNiceArray}
    \begin{bmatrix}
\vec{n}^*\\
\widetilde e_1^* \\
\widetilde e_2^*
\end{bmatrix}\notag\\
&=2(\widetilde  e_1^*, \widetilde e_2^*)_{|_{r=0}} E\frac{\p E}{\p r}(\widetilde  e_1^*, \widetilde e_2^*)_{|_{r=0}}^T\,.
    \end{align}
Since the terms involving $\vec{n}$-component all vanish in $L_{\vec{n}} g$, we can ignore the $\vec{n}$-component in calculating $L_{\vec{n}} g(e_i, e_j)$. We have at the point of interest:
\allowdisplaybreaks
\begin{equation}\label{e-1-e-2}
\begin{bmatrix}
e_1 \\
e_2
\end{bmatrix}=\begin{bmatrix}
\frac{\Theta+1}{2} & \frac{\Theta-1}{2} \\
\frac{\Theta-1}{2} & \frac{\Theta+1}{2}
\end{bmatrix} \begin{bmatrix}
\widehat e_1 \\
\widehat e_2
\end{bmatrix}=\begin{bmatrix}
\frac{\Theta+1}{2} & \frac{\Theta-1}{2} \\
\frac{\Theta-1}{2} & \frac{\Theta+1}{2}
\end{bmatrix}E^{-1}\begin{bmatrix}
\widetilde e_1 \\
\widetilde e_2
\end{bmatrix}.
\end{equation}
So we have 
\allowdisplaybreaks
\begin{align}\label{L-n-g-e-i-1}
&\left(L_{\vec{n}} g(e_i, e_j)\right) \\
=& \begin{bmatrix}
e_1 \\
e_2
\end{bmatrix}2(\widetilde  e_1^*, \widetilde e_2^*)E\frac{\p E}{\p r}(\widetilde  e_1^*, \widetilde e_2^*)^T(e_1, e_2) \notag\\
=& \begin{bmatrix}
\frac{\Theta+1}{2} & \frac{\Theta-1}{2} \\
\frac{\Theta-1}{2} & \frac{\Theta+1}{2}
\end{bmatrix}E^{-1}\begin{bmatrix}
\widetilde e_1 \\
\widetilde e_2
\end{bmatrix}2(\widetilde  e_1^*, \widetilde e_2^*)E\frac{\p E}{\p r}(\widetilde  e_1^*, \widetilde e_2^*)^T(\widetilde e_1, \widetilde e_2) E^{-1}\begin{bmatrix}
\frac{\Theta+1}{2} & \frac{\Theta-1}{2} \\
\frac{\Theta-1}{2} & \frac{\Theta+1}{2}
\end{bmatrix} \notag\\
=& \frac{1}{2}\begin{bmatrix}
\Theta+1 & \Theta-1 \\
\Theta-1 & \Theta+1
\end{bmatrix} E^{-1}\frac{\p E}{\p r}\begin{bmatrix}
\Theta+1 & \Theta-1 \\
\Theta-1 & \Theta+1
\end{bmatrix}\notag \\
=&  \frac{1}{2}\((\Theta+1)I_2+(\Theta-1)\begin{bmatrix}
0 & 1 \\
1 & 0
\end{bmatrix}\) \begin{bmatrix}
\alpha & \beta \\
\beta & \eta
\end{bmatrix}
\((\Theta+1)I_2+(\Theta-1)\begin{bmatrix}
0 & 1 \\
1 & 0
\end{bmatrix}\) \notag\\
=& \frac{1}{2}(\Theta+1)^2 \begin{bmatrix}
\alpha & \beta \\
\beta & \eta
\end{bmatrix}+\frac{1}{2}(\Theta^2-1)\begin{bmatrix}
\beta & \eta \\
\alpha & \beta
\end{bmatrix}+\frac{1}{2}(\Theta^2-1)\begin{bmatrix}
\beta & \alpha \\
\eta & \beta
\end{bmatrix}+\frac{1}{2}(\Theta-1)^2\begin{bmatrix}
\eta & \beta \\
\beta & \alpha
\end{bmatrix}.\notag
\end{align}
Hence we arrive at 
\allowdisplaybreaks
\begin{align}
&-L_{\vec{n}} g(e_i, e_j)A_{ij} \\
=&- {\rm Tr} \left((L_{\vec{n}} g(e_i, e_j))(A_{ij})\right) \notag\\
=& -\left(\frac{1}{2}(\Theta+1)^2(\alpha A_{11}+\beta A_{12}+\beta A_{12}+ \eta A_{22})\right. \notag\\
&\left.+\frac{1}{2}(\Theta^2-1)(\beta A_{11}+\eta A_{12}+\alpha A_{12}+ \beta A_{22}) \right.\notag\\
&\left.+\frac{1}{2}(\Theta^2-1)(\beta A_{11}+\alpha A_{12}+\eta A_{12}+ \beta A_{22}) \right.\notag\\
&\left.+\frac{1}{2}(\Theta-1)^2(\eta A_{11}+\beta A_{12}+\beta A_{12}+ \alpha A_{22})\right) \notag\\
=& -\left(\frac{1}{2}\Theta^2(A_{11}+2A_{12}+A_{22})(\alpha+2\beta+\eta)+\Theta(A_{11}-A_{22})(\alpha-\eta) \right.\notag\\
&\left. +\frac{1}{2}(\alpha-2\beta+\eta)(A_{11}-2A_{12}+A_{22}) \right)\notag\\
=& -\frac{1}{2}\Theta^3(\alpha+2\beta+\eta)^2-\Theta(\alpha-\eta)^2-\frac{1}{2}(\alpha-2\beta+\eta)\left[-\Theta(\alpha+2\beta+\eta)+2H\right],\notag
\end{align}
where we have used (\ref{eq:gradient=0-lin}) in the last step. 

\subsection{The terms $\sum_{i=1}^2\left[\frac{1}{2}(\overline\nabla_{\vec{\nu}} L_{\vec{n}} g)(e_i, e_i) -(\overline\nabla_{e_i}L_{\vec{n}} g)(\vec{\nu}, e_i)\right]$}
Next we will use \eqref{dual-e}, \eqref{coordinate-change-1} and calculate $\overline\nabla_{\vec{\nu}} L_{\vec{n}} g$ and $\overline\nabla_{e_i} L_{\vec{n}} g$ using $\overline\nabla_{\widetilde e_i} L_{\vec{n}} g$ and $\overline\nabla_{\vec{n}} L_{\vec{n}} g$. We first calculate $\overline\nabla_{\vec{n}} L_{\vec{n}}g$. Recall again that 
$$L_{\vec{n}} g=2(\widetilde  e_1^*, \widetilde e_2^*)E\frac{\p E}{\p r}(\widetilde  e_1^*, \widetilde e_2^*)^T.$$
Since $[\vec{n}, \widetilde e_i]=0$, we have $\overline\nabla_{\widetilde e_i} \vec{n}=\overline\nabla_{\vec{n}} \widetilde e_i$. Note also that we have \eqref{nabla-n-n}, i.e., $\overline\nabla_{\vec{n}} \vec{n}=0$, so (see \eqref{matrixF})
\begin{equation}\label{nabla-n-e12}
\begin{bmatrix}
\overline\nabla_{\vec{n}} \vec{n}\\
\overline\nabla_{\vec{n}} \widetilde e_1 \\
\overline\nabla_{\vec{n}} \widetilde e_2
\end{bmatrix}=\begin{bmatrix}
0\\
\overline\nabla_{\widetilde e_1} \vec{n} \\
\overline\nabla_{\widetilde e_2} \vec{n}
\end{bmatrix}=\begin{bmatrix}
0&0&0\\
0&\alpha & \beta \\
0& \beta & \eta
\end{bmatrix}\begin{bmatrix}
\vec{n}\\
\widetilde e_1 \\
\widetilde e_2
\end{bmatrix}=\begin{bNiceArray}{ccc}
0       & 0 & 0      \\
0       & \Block{2-2} {F} \\
0       &   &   
    \end{bNiceArray}\begin{bmatrix}
\vec{n}\\
\widetilde e_1 \\
\widetilde e_2
\end{bmatrix}.
\end{equation}
Thus we have
\begin{align*}
0=\overline\nabla_{\vec{n}}\(\begin{bmatrix}
\vec{n}\\
\widetilde e_1 \\
\widetilde e_2
\end{bmatrix}\circ(\vec{n}^*, \widetilde e_1^*, \widetilde e_2^*)\)=\begin{bmatrix}
0&0&0\\
0&\alpha & \beta \\
0& \beta & \eta
\end{bmatrix}+\begin{bmatrix}
\vec{n}\\
\widetilde e_1 \\
\widetilde e_2
\end{bmatrix}\circ\overline\nabla_{\vec{n}}(\vec{n}^*, \widetilde e_1^*, \widetilde e_2^*),
\end{align*}
where $\circ$ is the contraction between dual elements in dual spaces. Therefore, \begin{equation}
    \nabla_{\vec{n}} \vec{n}^* =0 
\end{equation}
and
\begin{equation*}
    \overline\nabla_{\vec{n}}(\vec{n}^*,\widetilde e_1^*, \widetilde e_2^*)=-(\vec{n}^*,\widetilde e_1^*, \widetilde e_2^*)\begin{bmatrix}
0&0&0\\
0&\alpha & \beta \\
0& \beta & \eta
\end{bmatrix}=-(\vec{n}^*, \widetilde e_1^*, \widetilde e_2^*)\begin{bNiceArray}{ccc}
0       & 0 & 0      \\
0       & \Block{2-2} {F} \\
0       &   &   
    \end{bNiceArray}.
\end{equation*}
Now we can use \eqref{L-n-g-1} to calculate $\overline\nabla_{\vec{n}} L_{\vec{n}} g$ as follows (again, since the terms involving the $\vec{n}$-component all vanish in $L_{\vec{n}} g$ and $\nabla_{\vec{n}} \vec{n}^*=0$, we can ignore the $\vec{n}$-component in calculating $\overline\nabla_{\vec{n}} L_{\vec{n}} g $):
\begin{align}\label{nabla-n-lng}
&\overline\nabla_{\vec{n}} L_{\vec{n}} g \notag\\
=& -2(\widetilde e_1^*, \widetilde e_2^*) F E\frac{\p E}{\p r} (\widetilde e_1^*, \widetilde e_2^*)^T-2(\widetilde e_1^*, \widetilde e_2^*)E\frac{\p E}{\p r}F(\widetilde e_1^*, \widetilde e_2^*)^T \notag\\
&+2(\widetilde  e_1^*, \widetilde e_2^*) \(\(\frac{\p E}{\p r}\)^2+E\frac{\p^2 E}{\p r^2}\)(\widetilde  e_1^*, \widetilde e_2^*)^T \notag\\
=& 2(\widetilde e_1^*, \widetilde e_2^*) \(E^2 - \(\frac{\p E}{\p r}\)^2\) (\widetilde e_1^*, \widetilde e_2^*)^T\notag\\
=& 2(1-\lambda^2(0))(\widetilde e_1^*, \widetilde e_2^*) I_2 (\widetilde e_1^*, \widetilde e_2^*)^T,
\end{align}
where we've also used Remark \eqref{remark1}, $\frac{\p^2 E}{\p r^2} = E$ and \eqref{dual-e}.

Now we calculate $\overline\nabla_{\widetilde e_i} L_{\vec{n}} g$. Let's set 
$$\overline\nabla_{\widetilde e_i} \widetilde e_j=- A_{ij}n+B_{ij}\widetilde e_1+C_{ij}\widetilde e_2 $$
at a point $(p,r)\in \Sigma(r).$ Note also that the Christoffel symbols $B_{ij}$ and $C_{ij}$ are of course symmetric since $0=[\widetilde e_i, \widetilde e_j]=\overline\nabla_{\widetilde e_i} \widetilde e_j-\overline\nabla_{\widetilde e_j} \widetilde e_i$ at the point $p\in\Sigma = \Sigma(0)$ and the local frame on $M$ is $\vec{n}$-parallel constructed from $\{\vec{n}, \widetilde e_1, \widetilde e_2\}|_{r=0}$.
By \eqref{frameconst1} we know that on $\Sigma =\Sigma(0)$:
\begin{equation}\label{BC1}
B_{ij} = C_{ij} =0, \quad \forall i, j =1,2.
\end{equation}
Note also that by \eqref{2ndff-1} and \eqref{Etimes-er} we have with respect to $\{\widetilde e_1, \widetilde e_2\}$
\allowdisplaybreaks
\begin{align*}
A_{ij} &= \left(A(\cdot, r)\right)_{ij} = E\frac{\partial E}{\partial r}\\
&=\begin{bmatrix}
\frac{1+\lambda^2(0)}{2}\sinh(2r)+a\cosh(2r) & b\cosh(2r) \\
b\cosh(2r) & \frac{1+\lambda^2(0)}{2}\sinh(2r)-a\cosh(2r)
\end{bmatrix}.
\end{align*}

By the formula for the Levi-Civita connection, we have at the point under consideration,  
$$\<\overline\nabla_{\widetilde e_1} \widetilde e_1, \widetilde e_1\>=\frac{1}{2}\widetilde e_1\<\widetilde e_1, \widetilde e_1\>, ~~\<\overline\nabla_{\widetilde e_1} \widetilde e_1, \widetilde e_2\>=\widetilde e_1\<\widetilde e_1, \widetilde e_2\>-\frac{1}{2}\widetilde e_2\<\widetilde e_1, \widetilde e_1\>,$$
$$\<\overline\nabla_{\widetilde e_1} \widetilde e_2, \widetilde e_1\>=\frac{1}{2}\widetilde e_2\<\widetilde e_1, \widetilde e_1\>, ~~\<\overline\nabla_{\widetilde e_1} \widetilde e_2, \widetilde e_2\>=\frac{1}{2}\widetilde e_1\<\widetilde e_2, \widetilde e_2\>,$$
$$\<\overline\nabla_{\widetilde e_2} \widetilde e_2, \widetilde e_2\>=\frac{1}{2}\widetilde e_2\<\widetilde e_2, \widetilde e_2\>,  ~~\<\overline\nabla_{\widetilde e_2} \widetilde e_2, \widetilde e_1\>=\widetilde e_2\<\widetilde e_1, \widetilde e_2\>-\frac{1}{2}\widetilde e_1\<\widetilde e_2, \widetilde e_2\>.$$
Namely, we have 
\begin{align}\label{chris1}
    \begin{bmatrix}
B_{11}  \\
C_{11} 
\end{bmatrix} =\frac{1}{2}g^{-1}\begin{bmatrix}
\partial_1 g_{11}  \\
2\partial_1 g_{12}-\partial_2 g_{11}
\end{bmatrix}
\end{align}
\begin{align}\label{chris2}
    \begin{bmatrix}
B_{12}  \\
C_{12} 
\end{bmatrix} =\frac{1}{2}g^{-1}\begin{bmatrix}
\partial_2 g_{11}  \\
\partial_1 g_{22}  
\end{bmatrix}
\end{align}
and
\begin{align}\label{chris3}
    \begin{bmatrix}
B_{22}  \\
C_{22} 
\end{bmatrix} =\frac{1}{2}g^{-1}\begin{bmatrix}
2\partial_2 g_{12}-\partial_1 g_{22}  \\
\partial_2 g_{22}
\end{bmatrix}\,,
\end{align}
where $\partial_i := \overline\nabla_{\widetilde e_i}$. Then one can use Remark \ref{exp-ex-1} to get explicit expressions for general $B_{ij}$ and $C_{ij}$ on $\Sigma(r)$:
\beq
 B_{11} = -\frac{2\sinh(r)(-m\cosh(r) + (a m + b n) \sinh(r))}{1+\lambda^2(0)+(1-\lambda^2(0))\cosh(2r)},
\eeq
\beq
 B_{12} = B_{21}= \frac{\sinh(r)(n \cosh(r) + (b m - a n) \sinh(r))}{1+\lambda^2(0)+(1-\lambda^2(0))\cosh(2r)},
\eeq
\beq
   B_{22} = \frac{2\sinh(r)(-m\cosh(r) + (a m + b n) \sinh(r))}{1+\lambda^2(0)+(1-\lambda^2(0))\cosh(2r)},
\eeq
\beq
   C_{11} = \frac{2\sinh(r)(n \cosh(r) + (-b m + a n) \sinh(r))}{1+\lambda^2(0)+(1-\lambda^2(0))\cosh(2r)},
\eeq
\beq
   C_{12}= C_{21} = -\frac{2\sinh(r)(m\cosh(r) + (a m + b n) \sinh(r))}{1+\lambda^2(0)+(1-\lambda^2(0))\cosh(2r)},
   \eeq
and
\beq
   C_{22} = -\frac{2\sinh(r) (n\cosh(r) + (-b m + a n) \sinh(r))}{1+\lambda^2(0)+(1-\lambda^2(0))\cosh(2r)}.
   \eeq
\begin{remark}
Clearly we have $B_{ij} = C_{ij} = 0$ on $\Sigma=\Sigma(0)$ as mentioned before. We have also used the Codazzi's equation
$$
m= \partial_1 a = \partial_1 (A_{\Sigma})_{11} = -\partial_1 (A_{\Sigma})_{22} = -\partial_2 (A_{\Sigma})_{12}=-\partial_2 b
$$
and
$$
n= \partial_2 a = \partial_2 (A_{\Sigma})_{11} = \partial_1 (A_{\Sigma})_{12} = \partial_1 b\,,
$$
to simply the expressions of $B_{ij}$'s and $C_{ij}$'s.
\end{remark}

In summary, we have (see \eqref{nabla-n-e12})
\begin{equation}\label{nabla-e}
\overline\nabla_{\widetilde e_1}\begin{bmatrix}
\vec{n}\\
\widetilde e_1 \\
\widetilde e_2
\end{bmatrix}=\begin{bmatrix}
0 & \alpha & \beta \\
-A_{11} & B_{11} & C_{11} \\
-A_{12} & B_{12} & C_{22}
\end{bmatrix}
\begin{bmatrix}
\vec{n}\\
\widetilde e_1 \\
\widetilde e_2
\end{bmatrix},
\end{equation}
\begin{equation}\label{nabla-e-1}
\overline\nabla_{\widetilde e_2}\begin{bmatrix}
\vec{n}\\
\widetilde e_1 \\
\widetilde e_2
\end{bmatrix}=\begin{bmatrix}
0 & \beta & \eta \\
-A_{21} & B_{21} & C_{21} \\
-A_{22} & B_{22} & C_{22}
\end{bmatrix}\begin{bmatrix}
\vec{n}\\
\widetilde e_1 \\
\widetilde e_2
\end{bmatrix},
\end{equation}
and so 
\begin{equation}\label{nabla-e-star}
\overline\nabla_{\widetilde e_1}(\vec{n}^*, \widetilde e_1^*, \widetilde e_2^*)
=-(\vec{n}^*, \widetilde e_1^*, \widetilde e_2^*)\begin{bmatrix}
0 & \alpha & \beta \\
-A_{11} & B_{11} & C_{11} \\
-A_{12} & B_{12} & C_{12}
\end{bmatrix},
\end{equation}
\begin{equation}\label{nabla-e-star-1}
\overline\nabla_{\widetilde e_2}(\vec{n}^*, \widetilde e_1^*, \widetilde e_2^*)
=-(\vec{n}^*, \widetilde e_1^*, \widetilde e_2^*)\begin{bmatrix}
0 & \beta & \eta \\
-A_{21} & B_{21} & C_{21} \\
-A_{22} & B_{22} & C_{22}
\end{bmatrix}.
\end{equation}

Recall again that (see \eqref{L-n-g-1})
\begin{align*}
    L_{\vec{n}} g
    &= 2(\vec{n}^*, \widetilde  e_1^*, \widetilde e_2^*)
    \begin{bNiceArray}{ccc}
0       & 0 & 0      \\
0       & \Block{2-2} {E\frac{\p E}{\p r}} \\
0       &   &   
    \end{bNiceArray}
    \begin{bmatrix}
\vec{n}^*\\
\widetilde e_1^* \\
\widetilde e_2^*
\end{bmatrix}\,,
    \end{align*}
    where
\begin{align*}
E\frac{\p E}{\p r}=[\cosh(r)I_2+\sinh(r)A_{\Sigma}]\cdot[\sinh(r)I_2+\cosh(r)A_{\Sigma}]\,.
\end{align*}
Thus we have
\begin{align*}
&\overline\nabla_{\widetilde e_1} L_{\vec{n}} g \\
=& 2\overline\nabla_{\widetilde e_1}(\vec{n}^*,\widetilde  e_1^*, \widetilde e_2^*)\begin{bNiceArray}{ccc}
0       & 0 & 0      \\
0       & \Block{2-2} {E\frac{\p E}{\p r}} \\
0       &   &   
    \end{bNiceArray}\begin{bmatrix}
\vec{n}^*\\
\widetilde e_1^* \\
\widetilde e_2^*
\end{bmatrix}  +2(\vec{n}^*,\widetilde  e_1^*, \widetilde e_2^*)\begin{bNiceArray}{ccc}
0       & 0 & 0      \\
0       & \Block{2-2} {E\frac{\p E}{\p r}} \\
0       &   &   
    \end{bNiceArray}\overline\nabla_{\widetilde e_1}\begin{bmatrix}
\vec{n}^*\\
\widetilde e_1^* \\
\widetilde e_2^*
\end{bmatrix}\\
&+2(\vec{n}^*,\widetilde  e_1^*, \widetilde e_2^*)\begin{bNiceArray}{ccc}
0       & \quad 0 & \quad 0      \\
0       & \Block{2-2} <\small>{\partial_1(E\frac{\p E}{\p r})} \\
0       &   &   
    \end{bNiceArray}\begin{bmatrix}
\vec{n}^*\\
\widetilde e_1^* \\
\widetilde e_2^*
\end{bmatrix} \\
=& -2(\vec{n}^*, \widetilde e_1^*, \widetilde e_2^*)\begin{bmatrix}
0& \alpha & \beta \\
-A_{11} & B_{11} & C_{11} \\
-A_{12} & B_{12} & C_{12}
\end{bmatrix}
\begin{bNiceArray}{ccc}
0       & 0 & 0      \\
0       & \Block{2-2} {E\frac{\p E}{\p r}} \\
0       &   &   
    \end{bNiceArray}\begin{bmatrix}
\vec{n}^*\\
\widetilde e_1^* \\
\widetilde e_2^*
\end{bmatrix} \\
&-2(\vec{n}^*, \widetilde  e_1^*, \widetilde e_2^*) \begin{bNiceArray}{ccc}
0       & 0 & 0      \\
0       & \Block{2-2} {E\frac{\p E}{\p r}} \\
0       &   &   
    \end{bNiceArray}\begin{bmatrix}
0& \alpha & \beta \\
-A_{11} & B_{11} & C_{11} \\
-A_{12} & B_{12} & C_{12}
\end{bmatrix}^T \begin{bmatrix}
\vec{n}^*\\
\widetilde e_1^* \\
\widetilde e_2^*
\end{bmatrix}\\
&+2(\vec{n}^*,\widetilde  e_1^*, \widetilde e_2^*)\begin{bNiceArray}{ccc}
0       & \quad 0 & \quad 0      \\
0       & \Block{2-2} <\small>{\partial_1(E\frac{\p E}{\p r})} \\
0       &   &   
    \end{bNiceArray}\begin{bmatrix}
\vec{n}^*\\
\widetilde e_1^* \\
\widetilde e_2^*
\end{bmatrix} \,,
\end{align*}

\begin{align*}
&\overline\nabla_{\widetilde e_2} L_{\vec{n}} g \\
=& 2\overline\nabla_{\widetilde e_2}(\vec{n}^*,\widetilde  e_1^*, \widetilde e_2^*)\begin{bNiceArray}{ccc}
0       & 0 & 0      \\
0       & \Block{2-2} {E\frac{\p E}{\p r}} \\
0       &   &   
    \end{bNiceArray}\begin{bmatrix}
\vec{n}^*\\
\widetilde e_1^* \\
\widetilde e_2^*
\end{bmatrix}  +2(\vec{n}^*,\widetilde  e_1^*, \widetilde e_2^*)\begin{bNiceArray}{ccc}
0       & 0 & 0      \\
0       & \Block{2-2} {E\frac{\p E}{\p r}} \\
0       &   &   
    \end{bNiceArray}\overline\nabla_{\widetilde e_2}\begin{bmatrix}
\vec{n}^*\\
\widetilde e_1^* \\
\widetilde e_2^*
\end{bmatrix}\\
&+2(\vec{n}^*,\widetilde  e_1^*, \widetilde e_2^*)\begin{bNiceArray}{ccc}
0       & \quad 0 & \quad 0      \\
0       & \Block{2-2} <\small>{\partial_2(E\frac{\p E}{\p r})} \\
0       &   &   
    \end{bNiceArray}\begin{bmatrix}
\vec{n}^*\\
\widetilde e_1^* \\
\widetilde e_2^*
\end{bmatrix} \\
=& -2(\vec{n}^*, \widetilde e_1^*, \widetilde e_2^*)\begin{bmatrix}
0& \beta & \eta \\
-A_{21} & B_{21} & C_{21} \\
-A_{22} & B_{22} & C_{22}
\end{bmatrix}
\begin{bNiceArray}{ccc}
0       & 0 & 0      \\
0       & \Block{2-2} {E\frac{\p E}{\p r}} \\
0       &   &   
    \end{bNiceArray}\begin{bmatrix}
\vec{n}^*\\
\widetilde e_1^* \\
\widetilde e_2^*
\end{bmatrix} \\
&-2(\vec{n}^*, \widetilde  e_1^*, \widetilde e_2^*) \begin{bNiceArray}{ccc}
0       & 0 & 0      \\
0       & \Block{2-2} {E\frac{\p E}{\p r}} \\
0       &   &   
    \end{bNiceArray}\begin{bmatrix}
0& \beta & \eta \\
-A_{21} & B_{21} & C_{21} \\
-A_{22} & B_{22} & C_{22}
\end{bmatrix}^T \begin{bmatrix}
\vec{n}^*\\
\widetilde e_1^* \\
\widetilde e_2^*
\end{bmatrix}\\
&+2(\vec{n}^*,\widetilde  e_1^*, \widetilde e_2^*)\begin{bNiceArray}{ccc}
0       & \quad 0 & \quad 0      \\
0       & \Block{2-2} <\small>{\partial_2(E\frac{\p E}{\p r})} \\
0       &   &   
    \end{bNiceArray}\begin{bmatrix}
\vec{n}^*\\
\widetilde e_1^* \\
\widetilde e_2^*
\end{bmatrix} \,,
\end{align*}
where $\partial_i =\overline\nabla_{\widetilde e_i}$.
Therefore using \eqref{dual-e} and \eqref{coordinate-change-1} we have
\begin{align*}
  \overline\nabla_{\begin{bmatrix}
\vec{\nu}\\
e_1 \\
e_2
\end{bmatrix}}L_{\vec{n}} g &= \begin{bmatrix}
\Theta & \frac{-\sqrt{1-\Theta^2}}{\sqrt{2}} & \frac{-\sqrt{1-\Theta^2}}{\sqrt{2}}\\
\frac{\sqrt{1-\Theta^2}}{\sqrt{2}} & \frac{\Theta+1}{2} & \frac{\Theta-1}{2} \\
\frac{\sqrt{1-\Theta^2}}{\sqrt{2}} & \frac{\Theta-1}{2} & \frac{\Theta+1}{2}
\end{bmatrix}\begin{bNiceArray}{ccc}
1       & \quad 0 & \quad 0      \\
0       & \Block{2-2} <>{E^{-1}} \\
0       &   &   
    \end{bNiceArray}  \overline\nabla_{\begin{bmatrix}
\vec{n}\\
\widetilde e_1 \\
\widetilde e_2
\end{bmatrix}}L_{\vec{n}} g\\
&=\begin{bmatrix}
\Theta & \frac{-\sqrt{1-\Theta^2}}{\sqrt{2}}(E^{-1}_{11}+E^{-1}_{12}) & \frac{-\sqrt{1-\Theta^2}}{\sqrt{2}}(E^{-1}_{22}+E^{-1}_{12})\\
\frac{\sqrt{1-\Theta^2}}{\sqrt{2}} & \frac{\Theta+1}{2}E^{-1}_{11} +\frac{\Theta-1}{2}E^{-1}_{12}  & \frac{\Theta+1}{2}E^{-1}_{12} +\frac{\Theta-1}{2}E^{-1}_{22} \\
\frac{\sqrt{1-\Theta^2}}{\sqrt{2}} & \frac{\Theta-1}{2}E^{-1}_{11} +\frac{\Theta+1}{2}E^{-1}_{12} & \frac{\Theta-1}{2}E^{-1}_{12} +\frac{\Theta+1}{2}E^{-1}_{22}
\end{bmatrix}
\overline\nabla_{\begin{bmatrix}
\vec{n}\\
\widetilde e_1 \\
\widetilde e_2
\end{bmatrix}}L_{\vec{n}} g,
\end{align*}
where (see \eqref{nabla-n-lng})
\begin{align}\label{nabla-n-lng-new}
&\overline\nabla_{\vec{n}} L_{\vec{n}} g = 2(1-\lambda^2(0))(\widetilde e_1^*, \widetilde e_2^*) I_2 (\widetilde e_1^*, \widetilde e_2^*)^T \notag\\
=& 2(\vec{n}^*,\widetilde  e_1^*, \widetilde e_2^*)
\begin{bNiceArray}{ccc}
0       &  0 &  0      \\
0       & 1-\lambda^2(0) & 0 \\
0       & 0  &   1-\lambda^2(0)
    \end{bNiceArray}\begin{bmatrix}
\vec{n}^*\\
\widetilde e_1^* \\
\widetilde e_2^*
\end{bmatrix}.
\end{align}
\vskip 3mm
\noindent In particular,
\allowdisplaybreaks
\begin{align*}
    &
    \frac{1}{2}\overline\nabla_{\vec{\nu}} L_{\vec{n}} g (e_1, e_1)\\
   = & \frac{1}{2}\overline\nabla_{\vec{\nu}} L_{\vec{n}} g \left(\begin{bmatrix}
\frac{\sqrt{1-\Theta^2}}{\sqrt{2}} & \frac{\Theta+1}{2}E^{-1}_{11} +\frac{\Theta-1}{2}E^{-1}_{12}  & \frac{\Theta+1}{2}E^{-1}_{12} +\frac{\Theta-1}{2}E^{-1}_{22}
\end{bmatrix}\begin{bmatrix}
\vec{n}\\
\widetilde e_1 \\
\widetilde e_2
\end{bmatrix},\right.\\
&\left.\begin{bmatrix}
\frac{\sqrt{1-\Theta^2}}{\sqrt{2}} & \frac{\Theta+1}{2}E^{-1}_{11} +\frac{\Theta-1}{2}E^{-1}_{12}  & \frac{\Theta+1}{2}E^{-1}_{12} +\frac{\Theta-1}{2}E^{-1}_{22}
\end{bmatrix}\begin{bmatrix}
\vec{n}\\
\widetilde e_1 \\
\widetilde e_2
\end{bmatrix}\right)\\
=&\frac{1}{2} \begin{bmatrix}\Theta & - \frac{\sqrt{1-\Theta^2}}{\sqrt{2}}(E^{-1}_{11} +E^{-1}_{12}) & - \frac{\sqrt{1-\Theta^2}}{\sqrt{2}}(E^{-1}_{22} +E^{-1}_{12})\end{bmatrix}
\begin{bmatrix} \overline\nabla_{\vec{n}} L_{\vec{n}} g \\ \overline\nabla_{\widetilde e_1} L_{\vec{n}} g \\
\overline\nabla_{\widetilde e_2} L_{\vec{n}} g \end{bmatrix}\\
&\left(\begin{bmatrix}
\frac{\sqrt{1-\Theta^2}}{\sqrt{2}} & \frac{\Theta+1}{2}E^{-1}_{11} +\frac{\Theta-1}{2}E^{-1}_{12}  & \frac{\Theta+1}{2}E^{-1}_{12} +\frac{\Theta-1}{2}E^{-1}_{22}
\end{bmatrix}\begin{bmatrix}
\vec{n}\\
\widetilde e_1 \\
\widetilde e_2
\end{bmatrix},\right.\\
&\left.\begin{bmatrix}
\frac{\sqrt{1-\Theta^2}}{\sqrt{2}} & \frac{\Theta+1}{2}E^{-1}_{11} +\frac{\Theta-1}{2}E^{-1}_{12}  & \frac{\Theta+1}{2}E^{-1}_{12} +\frac{\Theta-1}{2}E^{-1}_{22}
\end{bmatrix}\begin{bmatrix}
\vec{n}\\
\widetilde e_1 \\
\widetilde e_2
\end{bmatrix}\right)\\
=& \begin{bmatrix}
\frac{\sqrt{1-\Theta^2}}{\sqrt{2}} & \frac{\Theta+1}{2}E^{-1}_{11} +\frac{\Theta-1}{2}E^{-1}_{12}  & \frac{\Theta+1}{2}E^{-1}_{12} +\frac{\Theta-1}{2}E^{-1}_{22}
\end{bmatrix} \\
&\cdot\left\{\Theta\cdot\begin{bNiceArray}{ccc}
0       &  0 &  0      \\
0       & 1-\lambda^2(0) & 0 \\
0       & 0  &   1-\lambda^2(0)
    \end{bNiceArray}
  - \frac{\sqrt{1-\Theta^2}}{\sqrt{2}}(E^{-1}_{11} +E^{-1}_{12})\cdot\left(
  -\begin{bmatrix}
0& \alpha & \beta \\
-A_{11} & B_{11} & C_{11} \\
-A_{12} & B_{12} & C_{12}
\end{bmatrix} \right.\right.\\
&\left.\left.\cdot
\begin{bNiceArray}{ccc}
0       & 0 & 0      \\
0       & \Block{2-2} {E\frac{\p E}{\p r}} \\
0       &   &   
    \end{bNiceArray}
  -\begin{bNiceArray}{ccc}
0       & 0 & 0      \\
0       & \Block{2-2} {E\frac{\p E}{\p r}} \\
0       &   &   
    \end{bNiceArray}
    \begin{bmatrix}
0& \alpha & \beta \\
-A_{11} & B_{11} & C_{11} \\
-A_{12} & B_{12} & C_{12}
\end{bmatrix}^T 
+\begin{bNiceArray}{ccc}
0       & \quad 0 & \quad 0      \\
0       & \Block{2-2} <\small>{\partial_1(E\frac{\p E}{\p r})} \\
0       &   &   
    \end{bNiceArray}\right)\right.\\
    &\left.- \frac{\sqrt{1-\Theta^2}}{\sqrt{2}}(E^{-1}_{22} +E^{-1}_{12})\cdot 
\left(
  -\begin{bmatrix}
0& \beta & \eta \\
-A_{21} & B_{21} & C_{21} \\
-A_{22} & B_{22} & C_{22}
\end{bmatrix}\begin{bNiceArray}{ccc}
0       & 0 & 0      \\
0       & \Block{2-2} {E\frac{\p E}{\p r}} \\
0       &   &   
    \end{bNiceArray} \right.\right.\\
&\left.\left.
  -\begin{bNiceArray}{ccc}
0       & 0 & 0      \\
0       & \Block{2-2} {E\frac{\p E}{\p r}} \\
0       &   &   
    \end{bNiceArray}
    \begin{bmatrix}
0& \beta & \eta \\
-A_{21} & B_{21} & C_{21} \\
-A_{22} & B_{22} & C_{22}
\end{bmatrix}^T 
+\begin{bNiceArray}{ccc}
0       & \quad 0 & \quad 0      \\
0       & \Block{2-2} <\small>{\partial_2(E\frac{\p E}{\p r})} \\
0       &   &   
    \end{bNiceArray} 
    \right)
    \right\}\\
&\cdot
\begin{bmatrix}
\frac{\sqrt{1-\Theta^2}}{\sqrt{2}} & \frac{\Theta+1}{2}E^{-1}_{11} +\frac{\Theta-1}{2}E^{-1}_{12}  & \frac{\Theta+1}{2}E^{-1}_{12} +\frac{\Theta-1}{2}E^{-1}_{22}
\end{bmatrix}^T.
\end{align*}
\vskip 3mm
\noindent
Similarly,
\allowdisplaybreaks
\begin{align*}
    &
    \frac{1}{2}\overline\nabla_{\vec{\nu}} L_{\vec{n}} g (e_2, e_2)\\
   = & \frac{1}{2}\overline\nabla_{\vec{\nu}} L_{\vec{n}} g \left(\begin{bmatrix}
\frac{\sqrt{1-\Theta^2}}{\sqrt{2}} & \frac{\Theta-1}{2}E^{-1}_{11} +\frac{\Theta+1}{2}E^{-1}_{12}  & \frac{\Theta-1}{2}E^{-1}_{12} +\frac{\Theta+1}{2}E^{-1}_{22}
\end{bmatrix}\begin{bmatrix}
\vec{n}\\
\widetilde e_1 \\
\widetilde e_2
\end{bmatrix},\right.\\
&\left.\begin{bmatrix}
\frac{\sqrt{1-\Theta^2}}{\sqrt{2}} & \frac{\Theta-1}{2}E^{-1}_{11} +\frac{\Theta+1}{2}E^{-1}_{12}  & \frac{\Theta-1}{2}E^{-1}_{12} +\frac{\Theta+1}{2}E^{-1}_{22}
\end{bmatrix}\begin{bmatrix}
\vec{n}\\
\widetilde e_1 \\
\widetilde e_2
\end{bmatrix}\right)\\
=&\frac{1}{2} \begin{bmatrix}\Theta & - \frac{\sqrt{1-\Theta^2}}{\sqrt{2}}(E^{-1}_{11} +E^{-1}_{12}) & - \frac{\sqrt{1-\Theta^2}}{\sqrt{2}}(E^{-1}_{22} +E^{-1}_{12})\end{bmatrix}
\begin{bmatrix} \overline\nabla_{\vec{n}} L_{\vec{n}} g \\ \overline\nabla_{\widetilde e_1} L_{\vec{n}} g \\
\overline\nabla_{\widetilde e_2} L_{\vec{n}} g \end{bmatrix}\\
&\left(\begin{bmatrix}
\frac{\sqrt{1-\Theta^2}}{\sqrt{2}} & \frac{\Theta-1}{2}E^{-1}_{11} +\frac{\Theta+1}{2}E^{-1}_{12}  & \frac{\Theta-1}{2}E^{-1}_{12} +\frac{\Theta+1}{2}E^{-1}_{22}
\end{bmatrix}\begin{bmatrix}
\vec{n}\\
\widetilde e_1 \\
\widetilde e_2
\end{bmatrix},\right.\\
&\left.\begin{bmatrix}
\frac{\sqrt{1-\Theta^2}}{\sqrt{2}} & \frac{\Theta-1}{2}E^{-1}_{11} +\frac{\Theta+1}{2}E^{-1}_{12}  & \frac{\Theta-1}{2}E^{-1}_{12} +\frac{\Theta+1}{2}E^{-1}_{22}
\end{bmatrix}\begin{bmatrix}
\vec{n}\\
\widetilde e_1 \\
\widetilde e_2
\end{bmatrix}\right)\\
=& \begin{bmatrix}
\frac{\sqrt{1-\Theta^2}}{\sqrt{2}} & \frac{\Theta-1}{2}E^{-1}_{11} +\frac{\Theta+1}{2}E^{-1}_{12}  & \frac{\Theta-1}{2}E^{-1}_{12} +\frac{\Theta+1}{2}E^{-1}_{22}
\end{bmatrix} \\
&\cdot\left\{\Theta\cdot\begin{bNiceArray}{ccc}
0       &  0 &  0      \\
0       & 1-\lambda^2(0) & 0 \\
0       & 0  &   1-\lambda^2(0)
    \end{bNiceArray}
  - \frac{\sqrt{1-\Theta^2}}{\sqrt{2}}(E^{-1}_{11} +E^{-1}_{12})\cdot\left(
  -\begin{bmatrix}
0& \alpha & \beta \\
-A_{11} & B_{11} & C_{11} \\
-A_{12} & B_{12} & C_{12}
\end{bmatrix} \right.\right.\\
&\left.\left.\cdot
\begin{bNiceArray}{ccc}
0       & 0 & 0      \\
0       & \Block{2-2} {E\frac{\p E}{\p r}} \\
0       &   &   
    \end{bNiceArray}
  -\begin{bNiceArray}{ccc}
0       & 0 & 0      \\
0       & \Block{2-2} {E\frac{\p E}{\p r}} \\
0       &   &   
    \end{bNiceArray}
    \begin{bmatrix}
0& \alpha & \beta \\
-A_{11} & B_{11} & C_{11} \\
-A_{12} & B_{12} & C_{12}
\end{bmatrix}^T 
+\begin{bNiceArray}{ccc}
0       & \quad 0 & \quad 0      \\
0       & \Block{2-2} <\small>{\partial_1(E\frac{\p E}{\p r})} \\
0       &   &   
    \end{bNiceArray}\right)\right.\\
    &\left.- \frac{\sqrt{1-\Theta^2}}{\sqrt{2}}(E^{-1}_{22} +E^{-1}_{12})\cdot 
\left(
  -\begin{bmatrix}
0& \beta & \eta \\
-A_{21} & B_{21} & C_{21} \\
-A_{22} & B_{22} & C_{22}
\end{bmatrix}\begin{bNiceArray}{ccc}
0       & 0 & 0      \\
0       & \Block{2-2} {E\frac{\p E}{\p r}} \\
0       &   &   
    \end{bNiceArray} \right.\right.\\
&\left.\left.
  -\begin{bNiceArray}{ccc}
0       & 0 & 0      \\
0       & \Block{2-2} {E\frac{\p E}{\p r}} \\
0       &   &   
    \end{bNiceArray}
    \begin{bmatrix}
0& \beta & \eta \\
-A_{21} & B_{21} & C_{21} \\
-A_{22} & B_{22} & C_{22}
\end{bmatrix}^T 
+\begin{bNiceArray}{ccc}
0       & \quad 0 & \quad 0      \\
0       & \Block{2-2} <\small>{\partial_2(E\frac{\p E}{\p r})} \\
0       &   &   
    \end{bNiceArray} 
    \right)
    \right\}\\
&\cdot
\begin{bmatrix}
\frac{\sqrt{1-\Theta^2}}{\sqrt{2}} & \frac{\Theta-1}{2}E^{-1}_{11} +\frac{\Theta+1}{2}E^{-1}_{12}  & \frac{\Theta-1}{2}E^{-1}_{12} +\frac{\Theta+1}{2}E^{-1}_{22}
\end{bmatrix}^T.
\end{align*}

\vskip 3mm
\noindent
We also have
\allowdisplaybreaks
\begin{align*}
&(\overline\nabla_{e_1}L_{\vec{n}} g)(\vec{\nu}, e_1)\\
=&\begin{bmatrix}
\frac{\sqrt{1-\Theta^2}}{\sqrt{2}} & \frac{\Theta+1}{2}E^{-1}_{11} +\frac{\Theta-1}{2}E^{-1}_{12}  & \frac{\Theta+1}{2}E^{-1}_{12} +\frac{\Theta-1}{2}E^{-1}_{22}
\end{bmatrix}\begin{bmatrix} \overline\nabla_{\vec{n}} L_{\vec{n}} g \\ \overline\nabla_{\widetilde e_1} L_{\vec{n}} g \\
\overline\nabla_{\widetilde e_2} L_{\vec{n}} g \end{bmatrix}\\
&
\left(\begin{bmatrix}\Theta & - \frac{\sqrt{1-\Theta^2}}{\sqrt{2}}(E^{-1}_{11} +E^{-1}_{12}) & - \frac{\sqrt{1-\Theta^2}}{\sqrt{2}}(E^{-1}_{22} +E^{-1}_{12})\end{bmatrix}\begin{bmatrix}
\vec{n}\\
\widetilde e_1 \\
\widetilde e_2
\end{bmatrix},\right.\\
&\left.\begin{bmatrix}
\frac{\sqrt{1-\Theta^2}}{\sqrt{2}} & \frac{\Theta+1}{2}E^{-1}_{11} +\frac{\Theta-1}{2}E^{-1}_{12}  & \frac{\Theta+1}{2}E^{-1}_{12} +\frac{\Theta-1}{2}E^{-1}_{22}
\end{bmatrix}\begin{bmatrix}
\vec{n}\\
\widetilde e_1 \\
\widetilde e_2
\end{bmatrix}\right)\\
=& 2\begin{bmatrix}\Theta & - \frac{\sqrt{1-\Theta^2}}{\sqrt{2}}(E^{-1}_{11} +E^{-1}_{12}) & - \frac{\sqrt{1-\Theta^2}}{\sqrt{2}}(E^{-1}_{22} +E^{-1}_{12})\end{bmatrix} \\
&\cdot\left\{\frac{\sqrt{1-\Theta^2}}{\sqrt{2}} \cdot\begin{bNiceArray}{ccc}
0       &  0 &  0      \\
0       & 1-\lambda^2(0) & 0 \\
0       & 0  &   1-\lambda^2(0)
    \end{bNiceArray}
  +\left( \frac{\Theta+1}{2}E^{-1}_{11} +\frac{\Theta-1}{2}E^{-1}_{12}\right)\right.\\
  &\left.\cdot\left(
  -\begin{bmatrix}
0& \alpha & \beta \\
-A_{11} & B_{11} & C_{11} \\
-A_{12} & B_{12} & C_{12}
\end{bmatrix} \cdot
\begin{bNiceArray}{ccc}
0       & 0 & 0      \\
0       & \Block{2-2} {E\frac{\p E}{\p r}} \\
0       &   &   
    \end{bNiceArray}
  -\begin{bNiceArray}{ccc}
0       & 0 & 0      \\
0       & \Block{2-2} {E\frac{\p E}{\p r}} \\
0       &   &   
    \end{bNiceArray}
    \begin{bmatrix}
0& \alpha & \beta \\
-A_{11} & B_{11} & C_{11} \\
-A_{12} & B_{12} & C_{12}
\end{bmatrix}^T \right.\right.\\
&\left.\left.
+\begin{bNiceArray}{ccc}
0       & \quad 0 & \quad 0      \\
0       & \Block{2-2} <\small>{\partial_1(E\frac{\p E}{\p r})} \\
0       &   &   
    \end{bNiceArray}\right) +\left(\frac{\Theta+1}{2}E^{-1}_{12} +\frac{\Theta-1}{2}E^{-1}_{22}\right)\cdot 
\left(
  -\begin{bmatrix}
0& \beta & \eta \\
-A_{21} & B_{21} & C_{21} \\
-A_{22} & B_{22} & C_{22}
\end{bmatrix}\right.\right.\\
&\left.\left.\cdot \begin{bNiceArray}{ccc}
0       & 0 & 0      \\
0       & \Block{2-2} {E\frac{\p E}{\p r}} \\
0       &   &   
    \end{bNiceArray} 
  -\begin{bNiceArray}{ccc}
0       & 0 & 0      \\
0       & \Block{2-2} {E\frac{\p E}{\p r}} \\
0       &   &   
    \end{bNiceArray}
    \begin{bmatrix}
0& \beta & \eta \\
-A_{21} & B_{21} & C_{21} \\
-A_{22} & B_{22} & C_{22}
\end{bmatrix}^T 
+\begin{bNiceArray}{ccc}
0       & \quad 0 & \quad 0      \\
0       & \Block{2-2} <\small>{\partial_2(E\frac{\p E}{\p r})} \\
0       &   &   
    \end{bNiceArray} 
    \right)
    \right\}\\
&\cdot
\begin{bmatrix}
\frac{\sqrt{1-\Theta^2}}{\sqrt{2}} & \frac{\Theta+1}{2}E^{-1}_{11} +\frac{\Theta-1}{2}E^{-1}_{12}  & \frac{\Theta+1}{2}E^{-1}_{12} +\frac{\Theta-1}{2}E^{-1}_{22}
\end{bmatrix}^T.
\end{align*}

\vskip 3mm
\noindent Similarly,
\allowdisplaybreaks
\begin{align*}
&(\overline\nabla_{e_2}L_{\vec{n}} g)(\vec{\nu}, e_2)\\
=&\begin{bmatrix}
\frac{\sqrt{1-\Theta^2}}{\sqrt{2}} & \frac{\Theta-1}{2}E^{-1}_{11} +\frac{\Theta+1}{2}E^{-1}_{12}  & \frac{\Theta-1}{2}E^{-1}_{12} +\frac{\Theta+1}{2}E^{-1}_{22}
\end{bmatrix}\begin{bmatrix} \overline\nabla_{\vec{n}} L_{\vec{n}} g \\ \overline\nabla_{\widetilde e_1} L_{\vec{n}} g \\
\overline\nabla_{\widetilde e_2} L_{\vec{n}} g \end{bmatrix}\\
&
\left(\begin{bmatrix}\Theta & - \frac{\sqrt{1-\Theta^2}}{\sqrt{2}}(E^{-1}_{11} +E^{-1}_{12}) & - \frac{\sqrt{1-\Theta^2}}{\sqrt{2}}(E^{-1}_{22} +E^{-1}_{12})\end{bmatrix}\begin{bmatrix}
\vec{n}\\
\widetilde e_1 \\
\widetilde e_2
\end{bmatrix},\right.\\
&\left.\begin{bmatrix}
\frac{\sqrt{1-\Theta^2}}{\sqrt{2}} & \frac{\Theta-1}{2}E^{-1}_{11} +\frac{\Theta+1}{2}E^{-1}_{12}  & \frac{\Theta-1}{2}E^{-1}_{12} +\frac{\Theta+1}{2}E^{-1}_{22}
\end{bmatrix}\begin{bmatrix}
\vec{n}\\
\widetilde e_1 \\
\widetilde e_2
\end{bmatrix}\right)\\
=& 2\begin{bmatrix}\Theta & - \frac{\sqrt{1-\Theta^2}}{\sqrt{2}}(E^{-1}_{11} +E^{-1}_{12}) & - \frac{\sqrt{1-\Theta^2}}{\sqrt{2}}(E^{-1}_{22} +E^{-1}_{12})\end{bmatrix} \\
&\cdot\left\{\frac{\sqrt{1-\Theta^2}}{\sqrt{2}} \cdot\begin{bNiceArray}{ccc}
0       &  0 &  0      \\
0       & 1-\lambda^2(0) & 0 \\
0       & 0  &   1-\lambda^2(0)
    \end{bNiceArray}
  +\left( \frac{\Theta-1}{2}E^{-1}_{11} +\frac{\Theta+1}{2}E^{-1}_{12}\right)\right.\\
  &\left.\cdot\left(
  -\begin{bmatrix}
0& \alpha & \beta \\
-A_{11} & B_{11} & C_{11} \\
-A_{12} & B_{12} & C_{12}
\end{bmatrix} \cdot
\begin{bNiceArray}{ccc}
0       & 0 & 0      \\
0       & \Block{2-2} {E\frac{\p E}{\p r}} \\
0       &   &   
    \end{bNiceArray}
  -\begin{bNiceArray}{ccc}
0       & 0 & 0      \\
0       & \Block{2-2} {E\frac{\p E}{\p r}} \\
0       &   &   
    \end{bNiceArray}
    \begin{bmatrix}
0& \alpha & \beta \\
-A_{11} & B_{11} & C_{11} \\
-A_{12} & B_{12} & C_{12}
\end{bmatrix}^T \right.\right.\\
&
\left.\left.+\begin{bNiceArray}{ccc}
0       & \quad 0 & \quad 0      \\
0       & \Block{2-2} <\small>{\partial_1(E\frac{\p E}{\p r})} \\
0       &   &   
    \end{bNiceArray}\right) +\left(\frac{\Theta-1}{2}E^{-1}_{12} +\frac{\Theta+1}{2}E^{-1}_{22}\right)\cdot 
\left(
  -\begin{bmatrix}
0& \beta & \eta \\
-A_{21} & B_{21} & C_{21} \\
-A_{22} & B_{22} & C_{22}
\end{bmatrix}\right.\right.\\
&\left.\left.\cdot
\begin{bNiceArray}{ccc}
0       & 0 & 0      \\
0       & \Block{2-2} {E\frac{\p E}{\p r}} \\
0       &   &   
    \end{bNiceArray}
  -\begin{bNiceArray}{ccc}
0       & 0 & 0      \\
0       & \Block{2-2} {E\frac{\p E}{\p r}} \\
0       &   &   
    \end{bNiceArray}
    \begin{bmatrix}
0& \beta & \eta \\
-A_{21} & B_{21} & C_{21} \\
-A_{22} & B_{22} & C_{22}
\end{bmatrix}^T 
+\begin{bNiceArray}{ccc}
0       & \quad 0 & \quad 0      \\
0       & \Block{2-2} <\small>{\partial_2(E\frac{\p E}{\p r})} \\
0       &   &   
    \end{bNiceArray} 
    \right)
    \right\}\\
&\cdot
\begin{bmatrix}
\frac{\sqrt{1-\Theta^2}}{\sqrt{2}} & \frac{\Theta-1}{2}E^{-1}_{11} +\frac{\Theta+1}{2}E^{-1}_{12}  & \frac{\Theta-1}{2}E^{-1}_{12} +\frac{\Theta+1}{2}E^{-1}_{22}
\end{bmatrix}^T.
\end{align*}

\vskip 3mm
\noindent 
Now set
\allowdisplaybreaks
\begin{align*}
P :&= \begin{bNiceArray}{ccc}
0       &  0 &  0      \\
0       & 1-\lambda^2(0) & 0 \\
0       & 0  &   1-\lambda^2(0)
    \end{bNiceArray}\\
Q_1:&= \begin{bmatrix}
0& \alpha & \beta \\
-A_{11} & B_{11} & C_{11} \\
-A_{12} & B_{12} & C_{12}
\end{bmatrix} \cdot
\begin{bNiceArray}{ccc}
0       & 0 & 0      \\
0       & \Block{2-2} {E\frac{\p E}{\p r}} \\
0       &   &   
    \end{bNiceArray}  \\
    R_1:&=\begin{bNiceArray}{ccc}
0       & \quad 0 & \quad 0      \\
0       & \Block{2-2} <\small>{\partial_1(E\frac{\p E}{\p r})} \\
0       &   &   
    \end{bNiceArray}\\
    Q_2:&= \begin{bmatrix}
0& \beta & \eta \\
-A_{21} & B_{21} & C_{21} \\
-A_{22} & B_{22} & C_{22}
\end{bmatrix} \cdot
\begin{bNiceArray}{ccc}
0       & 0 & 0      \\
0       & \Block{2-2} {E\frac{\p E}{\p r}} \\
0       &   &   
    \end{bNiceArray}  \\
    R_2:&=\begin{bNiceArray}{ccc}
0       & \quad 0 & \quad 0      \\
0       & \Block{2-2} <\small>{\partial_2(E\frac{\p E}{\p r})} \\
0       &   &   
    \end{bNiceArray}.
\end{align*}

\vskip 3mm 
\noindent
Combining all these, we have
\allowdisplaybreaks
\begin{align*}
&\sum_{i=1}^2\left[\frac{1}{2}(\overline\nabla_{\vec{\nu}} L_{\vec{n}} g)(e_i, e_i) -(\overline\nabla_{e_i}L_{\vec{n}} g)(\vec{\nu}, e_i)\right]\\
=& \begin{bmatrix}
\frac{\sqrt{1-\Theta^2}}{\sqrt{2}} & \frac{\Theta+1}{2}E^{-1}_{11} +\frac{\Theta-1}{2}E^{-1}_{12}  & \frac{\Theta+1}{2}E^{-1}_{12} +\frac{\Theta-1}{2}E^{-1}_{22}
\end{bmatrix} \\
&\cdot\left\{\Theta\cdot\begin{bNiceArray}{ccc}
0       &  0 &  0      \\
0       & 1-\lambda^2(0) & 0 \\
0       & 0  &   1-\lambda^2(0)
    \end{bNiceArray}
  - \frac{\sqrt{1-\Theta^2}}{\sqrt{2}}(E^{-1}_{11} +E^{-1}_{12})\right.\\
 & \left.\cdot\left(
  -\begin{bmatrix}
0& \alpha & \beta \\
-A_{11} & B_{11} & C_{11} \\
-A_{12} & B_{12} & C_{12}
\end{bmatrix}
\begin{bNiceArray}{ccc}
0       & 0 & 0      \\
0       & \Block{2-2} {E\frac{\p E}{\p r}} \\
0       &   &   
    \end{bNiceArray}
  -\begin{bNiceArray}{ccc}
0       & 0 & 0      \\
0       & \Block{2-2} {E\frac{\p E}{\p r}} \\
0       &   &   
    \end{bNiceArray}
    \begin{bmatrix}
0& \alpha & \beta \\
-A_{11} & B_{11} & C_{11} \\
-A_{12} & B_{12} & C_{12}
\end{bmatrix}^T \right.\right.\\
&\left.\left.
+\begin{bNiceArray}{ccc}
0       & \quad 0 & \quad 0      \\
0       & \Block{2-2} <\small>{\partial_1(E\frac{\p E}{\p r})} \\
0       &   &   
    \end{bNiceArray}\right)
    - \frac{\sqrt{1-\Theta^2}}{\sqrt{2}}(E^{-1}_{22} +E^{-1}_{12})\cdot 
\left(
  -\begin{bmatrix}
0& \beta & \eta \\
-A_{21} & B_{21} & C_{21} \\
-A_{22} & B_{22} & C_{22}
\end{bmatrix}\begin{bNiceArray}{ccc}
0       & 0 & 0      \\
0       & \Block{2-2} {E\frac{\p E}{\p r}} \\
0       &   &   
    \end{bNiceArray} \right.\right.\\
&\left.\left.
  -\begin{bNiceArray}{ccc}
0       & 0 & 0      \\
0       & \Block{2-2} {E\frac{\p E}{\p r}} \\
0       &   &   
    \end{bNiceArray}
    \begin{bmatrix}
0& \beta & \eta \\
-A_{21} & B_{21} & C_{21} \\
-A_{22} & B_{22} & C_{22}
\end{bmatrix}^T 
+\begin{bNiceArray}{ccc}
0       & \quad 0 & \quad 0      \\
0       & \Block{2-2} <\small>{\partial_2(E\frac{\p E}{\p r})} \\
0       &   &   
    \end{bNiceArray} 
    \right)
    \right\}\\
&\cdot
\begin{bmatrix}
\frac{\sqrt{1-\Theta^2}}{\sqrt{2}} & \frac{\Theta+1}{2}E^{-1}_{11} +\frac{\Theta-1}{2}E^{-1}_{12}  & \frac{\Theta+1}{2}E^{-1}_{12} +\frac{\Theta-1}{2}E^{-1}_{22}
\end{bmatrix}^T\\
&+ \begin{bmatrix}
\frac{\sqrt{1-\Theta^2}}{\sqrt{2}} & \frac{\Theta-1}{2}E^{-1}_{11} +\frac{\Theta+1}{2}E^{-1}_{12}  & \frac{\Theta-1}{2}E^{-1}_{12} +\frac{\Theta+1}{2}E^{-1}_{22}
\end{bmatrix} \\
&\cdot\left\{\Theta\cdot\begin{bNiceArray}{ccc}
0       &  0 &  0      \\
0       & 1-\lambda^2(0) & 0 \\
0       & 0  &   1-\lambda^2(0)
    \end{bNiceArray}
  - \frac{\sqrt{1-\Theta^2}}{\sqrt{2}}(E^{-1}_{11} +E^{-1}_{12})\cdot\left(
  -\begin{bmatrix}
0& \alpha & \beta \\
-A_{11} & B_{11} & C_{11} \\
-A_{12} & B_{12} & C_{12}
\end{bmatrix} \right.\right.\\
&\left.\left.\cdot
\begin{bNiceArray}{ccc}
0       & 0 & 0      \\
0       & \Block{2-2} {E\frac{\p E}{\p r}} \\
0       &   &   
    \end{bNiceArray}
  -\begin{bNiceArray}{ccc}
0       & 0 & 0      \\
0       & \Block{2-2} {E\frac{\p E}{\p r}} \\
0       &   &   
    \end{bNiceArray}
    \begin{bmatrix}
0& \alpha & \beta \\
-A_{11} & B_{11} & C_{11} \\
-A_{12} & B_{12} & C_{12}
\end{bmatrix}^T 
+\begin{bNiceArray}{ccc}
0       & \quad 0 & \quad 0      \\
0       & \Block{2-2} <\small>{\partial_1(E\frac{\p E}{\p r})} \\
0       &   &   
    \end{bNiceArray}\right)\right.\\
    &\left.- \frac{\sqrt{1-\Theta^2}}{\sqrt{2}}(E^{-1}_{22} +E^{-1}_{12})\cdot 
\left(
  -\begin{bmatrix}
0& \beta & \eta \\
-A_{21} & B_{21} & C_{21} \\
-A_{22} & B_{22} & C_{22}
\end{bmatrix}\begin{bNiceArray}{ccc}
0       & 0 & 0      \\
0       & \Block{2-2} {E\frac{\p E}{\p r}} \\
0       &   &   
    \end{bNiceArray} \right.\right.\\
&\left.\left.
  -\begin{bNiceArray}{ccc}
0       & 0 & 0      \\
0       & \Block{2-2} {E\frac{\p E}{\p r}} \\
0       &   &   
    \end{bNiceArray}
    \begin{bmatrix}
0& \beta & \eta \\
-A_{21} & B_{21} & C_{21} \\
-A_{22} & B_{22} & C_{22}
\end{bmatrix}^T 
+\begin{bNiceArray}{ccc}
0       & \quad 0 & \quad 0      \\
0       & \Block{2-2} <\small>{\partial_2(E\frac{\p E}{\p r})} \\
0       &   &   
    \end{bNiceArray} 
    \right)
    \right\}\\
&\cdot
\begin{bmatrix}
\frac{\sqrt{1-\Theta^2}}{\sqrt{2}} & \frac{\Theta-1}{2}E^{-1}_{11} +\frac{\Theta+1}{2}E^{-1}_{12}  & \frac{\Theta-1}{2}E^{-1}_{12} +\frac{\Theta+1}{2}E^{-1}_{22}
\end{bmatrix}^T\\
&- 2\begin{bmatrix}\Theta & - \frac{\sqrt{1-\Theta^2}}{\sqrt{2}}(E^{-1}_{11} +E^{-1}_{12}) & - \frac{\sqrt{1-\Theta^2}}{\sqrt{2}}(E^{-1}_{22} +E^{-1}_{12})\end{bmatrix} \\
&\cdot\left\{\frac{\sqrt{1-\Theta^2}}{\sqrt{2}} \cdot\begin{bNiceArray}{ccc}
0       &  0 &  0      \\
0       & 1-\lambda^2(0) & 0 \\
0       & 0  &   1-\lambda^2(0)
    \end{bNiceArray}
  +\left( \frac{\Theta+1}{2}E^{-1}_{11} +\frac{\Theta-1}{2}E^{-1}_{12}\right)\right.\\
  & \left. \cdot\left(
  -\begin{bmatrix}
0& \alpha & \beta \\
-A_{11} & B_{11} & C_{11} \\
-A_{12} & B_{12} & C_{12}
\end{bmatrix}
\begin{bNiceArray}{ccc}
0       & 0 & 0      \\
0       & \Block{2-2} {E\frac{\p E}{\p r}} \\
0       &   &   
    \end{bNiceArray}
  -\begin{bNiceArray}{ccc}
0       & 0 & 0      \\
0       & \Block{2-2} {E\frac{\p E}{\p r}} \\
0       &   &   
    \end{bNiceArray}
    \begin{bmatrix}
0& \alpha & \beta \\
-A_{11} & B_{11} & C_{11} \\
-A_{12} & B_{12} & C_{12}
\end{bmatrix}^T \right.\right.\\
&\left.\left.
+\begin{bNiceArray}{ccc}
0       & \quad 0 & \quad 0      \\
0       & \Block{2-2} <\small>{\partial_1(E\frac{\p E}{\p r})} \\
0       &   &   
    \end{bNiceArray}\right) +\left(\frac{\Theta+1}{2}E^{-1}_{12} +\frac{\Theta-1}{2}E^{-1}_{22}\right)\cdot 
\left(
  -\begin{bmatrix}
0& \beta & \eta \\
-A_{21} & B_{21} & C_{21} \\
-A_{22} & B_{22} & C_{22}
\end{bmatrix}\right.\right.\\
&\left.\left.\begin{bNiceArray}{ccc}
0       & 0 & 0      \\
0       & \Block{2-2} {E\frac{\p E}{\p r}} \\
0       &   &   
    \end{bNiceArray}
  -\begin{bNiceArray}{ccc}
0       & 0 & 0      \\
0       & \Block{2-2} {E\frac{\p E}{\p r}} \\
0       &   &   
    \end{bNiceArray}
    \begin{bmatrix}
0& \beta & \eta \\
-A_{21} & B_{21} & C_{21} \\
-A_{22} & B_{22} & C_{22}
\end{bmatrix}^T 
+\begin{bNiceArray}{ccc}
0       & \quad 0 & \quad 0      \\
0       & \Block{2-2} <\small>{\partial_2(E\frac{\p E}{\p r})} \\
0       &   &   
    \end{bNiceArray} 
    \right)
    \right\}\\
&\cdot
\begin{bmatrix}
\frac{\sqrt{1-\Theta^2}}{\sqrt{2}} & \frac{\Theta+1}{2}E^{-1}_{11} +\frac{\Theta-1}{2}E^{-1}_{12}  & \frac{\Theta+1}{2}E^{-1}_{12} +\frac{\Theta-1}{2}E^{-1}_{22}
\end{bmatrix}^T\\
&-2\begin{bmatrix}\Theta & - \frac{\sqrt{1-\Theta^2}}{\sqrt{2}}(E^{-1}_{11} +E^{-1}_{12}) & - \frac{\sqrt{1-\Theta^2}}{\sqrt{2}}(E^{-1}_{22} +E^{-1}_{12})\end{bmatrix} \\
&\cdot\left\{\frac{\sqrt{1-\Theta^2}}{\sqrt{2}} \cdot\begin{bNiceArray}{ccc}
0       &  0 &  0      \\
0       & 1-\lambda^2(0) & 0 \\
0       & 0  &   1-\lambda^2(0)
    \end{bNiceArray}
  +\left( \frac{\Theta-1}{2}E^{-1}_{11} +\frac{\Theta+1}{2}E^{-1}_{12}\right)\right.\\
 &\left. \cdot\left(
  -\begin{bmatrix}
0& \alpha & \beta \\
-A_{11} & B_{11} & C_{11} \\
-A_{12} & B_{12} & C_{12}
\end{bmatrix} 
\begin{bNiceArray}{ccc}
0       & 0 & 0      \\
0       & \Block{2-2} {E\frac{\p E}{\p r}} \\
0       &   &   
    \end{bNiceArray}
  -\begin{bNiceArray}{ccc}
0       & 0 & 0      \\
0       & \Block{2-2} {E\frac{\p E}{\p r}} \\
0       &   &   
    \end{bNiceArray}
    \begin{bmatrix}
0& \alpha & \beta \\
-A_{11} & B_{11} & C_{11} \\
-A_{12} & B_{12} & C_{12}
\end{bmatrix}^T \right.\right.\\
&\left.\left.+\begin{bNiceArray}{ccc}
0       & \quad 0 & \quad 0      \\
0       & \Block{2-2} <\small>{\partial_1(E\frac{\p E}{\p r})} \\
0       &   &   
    \end{bNiceArray}\right)
    +\left(\frac{\Theta-1}{2}E^{-1}_{12} +\frac{\Theta+1}{2}E^{-1}_{22}\right)\cdot 
\left(
  -\begin{bmatrix}
0& \beta & \eta \\
-A_{21} & B_{21} & C_{21} \\
-A_{22} & B_{22} & C_{22}
\end{bmatrix}\right.\right.\\
&\left.\left.\begin{bNiceArray}{ccc}
0       & 0 & 0      \\
0       & \Block{2-2} {E\frac{\p E}{\p r}} \\
0       &   &   
    \end{bNiceArray}
  -\begin{bNiceArray}{ccc}
0       & 0 & 0      \\
0       & \Block{2-2} {E\frac{\p E}{\p r}} \\
0       &   &   
    \end{bNiceArray}
    \begin{bmatrix}
0& \beta & \eta \\
-A_{21} & B_{21} & C_{21} \\
-A_{22} & B_{22} & C_{22}
\end{bmatrix}^T 
+\begin{bNiceArray}{ccc}
0       & \quad 0 & \quad 0      \\
0       & \Block{2-2} <\small>{\partial_2(E\frac{\p E}{\p r})} \\
0       &   &   
    \end{bNiceArray} 
    \right)
    \right\}\\
&\cdot
\begin{bmatrix}
\frac{\sqrt{1-\Theta^2}}{\sqrt{2}} & \frac{\Theta-1}{2}E^{-1}_{11} +\frac{\Theta+1}{2}E^{-1}_{12}  & \frac{\Theta-1}{2}E^{-1}_{12} +\frac{\Theta+1}{2}E^{-1}_{22}
\end{bmatrix}^T\\
=& \begin{bmatrix}
\frac{\sqrt{1-\Theta^2}}{\sqrt{2}} \\ \frac{\Theta+1}{2}E^{-1}_{11} +\frac{\Theta-1}{2}E^{-1}_{12}  \\ \frac{\Theta+1}{2}E^{-1}_{12} +\frac{\Theta-1}{2}E^{-1}_{22}
\end{bmatrix}^T \cdot\left\{\Theta\cdot P
  + \sqrt{2(1-\Theta^2)}(E^{-1}_{11} +E^{-1}_{12})\cdot\frac{
  Q_1 +Q_1^T-R_1}{2}\right.\\
    &\left.+ \sqrt{2(1-\Theta^2)}(E^{-1}_{22} +E^{-1}_{12})\cdot\frac{
  Q_2 +Q_2^T-R_2}{2}\right\}\cdot
\begin{bmatrix}
\frac{\sqrt{1-\Theta^2}}{\sqrt{2}} \\ \frac{\Theta+1}{2}E^{-1}_{11} +\frac{\Theta-1}{2}E^{-1}_{12}  \\ \frac{\Theta+1}{2}E^{-1}_{12} +\frac{\Theta-1}{2}E^{-1}_{22}
\end{bmatrix}\\
& +\begin{bmatrix}
\frac{\sqrt{1-\Theta^2}}{\sqrt{2}} \\ \frac{\Theta-1}{2}E^{-1}_{11} +\frac{\Theta+1}{2}E^{-1}_{12}  \\ \frac{\Theta-1}{2}E^{-1}_{12} +\frac{\Theta+1}{2}E^{-1}_{22}
\end{bmatrix}^T \cdot\left\{\Theta\cdot P
  + \sqrt{2(1-\Theta^2)}(E^{-1}_{11} +E^{-1}_{12})\cdot\frac{
  Q_1 +Q_1^T-R_1}{2}\right.\\
    &\left.+ \sqrt{2(1-\Theta^2)}(E^{-1}_{22} +E^{-1}_{12})\cdot\frac{
  Q_2 +Q_2^T-R_2}{2}\right\}\cdot
\begin{bmatrix}
\frac{\sqrt{1-\Theta^2}}{\sqrt{2}} \\ \frac{\Theta-1}{2}E^{-1}_{11} +\frac{\Theta+1}{2}E^{-1}_{12}  \\ \frac{\Theta-1}{2}E^{-1}_{12} +\frac{\Theta+1}{2}E^{-1}_{22}
\end{bmatrix}\\
&+\begin{bmatrix}-2\Theta \\  \sqrt{2(1-\Theta^2)}(E^{-1}_{11} +E^{-1}_{12}) \\  \sqrt{2(1-\Theta^2)}(E^{-1}_{22} +E^{-1}_{12})\end{bmatrix}^T 
\cdot\Big\{\frac{\sqrt{1-\Theta^2}}{\sqrt{2}} \cdot P
  -\left( (\Theta+1)E^{-1}_{11} +(\Theta-1)E^{-1}_{12}\right)\\
  &\cdot \frac{
  Q_1 +Q_1^T-R_1}{2}
  -\left((\Theta+1)E^{-1}_{12} +(\Theta-1)E^{-1}_{22}\right)\cdot \frac{
  Q_2 +Q_2^T-R_2}{2}
\Big\}\cdot
\begin{bmatrix}
\frac{\sqrt{1-\Theta^2}}{\sqrt{2}} \\ \frac{\Theta+1}{2}E^{-1}_{11} +\frac{\Theta-1}{2}E^{-1}_{12}  \\ \frac{\Theta+1}{2}E^{-1}_{12} +\frac{\Theta-1}{2}E^{-1}_{22}
\end{bmatrix}\\
&+\begin{bmatrix}-2\Theta \\  \sqrt{2(1-\Theta^2)}(E^{-1}_{11} +E^{-1}_{12}) \\  \sqrt{2(1-\Theta^2)}(E^{-1}_{22} +E^{-1}_{12})\end{bmatrix}^T
\cdot\Big\{\frac{\sqrt{1-\Theta^2}}{\sqrt{2}} \cdot P
  -\left( (\Theta-1)E^{-1}_{11} +(\Theta+1)E^{-1}_{12}\right)\\
  &\cdot \frac{
  Q_1 +Q_1^T-R_1}{2}
   -\left((\Theta-1)E^{-1}_{12} +(\Theta+1)E^{-1}_{22}\right)\cdot  \frac{
  Q_2 +Q_2^T-R_2}{2}
    \Big\} \cdot
\begin{bmatrix}
\frac{\sqrt{1-\Theta^2}}{\sqrt{2}} \\ \frac{\Theta-1}{2}E^{-1}_{11} +\frac{\Theta+1}{2}E^{-1}_{12}  \\ \frac{\Theta-1}{2}E^{-1}_{12} +\frac{\Theta+1}{2}E^{-1}_{22}
\end{bmatrix}\\
=& {\text{Tr}}
\left\{\left[
U
\cdot T\cdot P +\frac{1}{2}
V \cdot T
\cdot (Q_1+Q_1^T-R_1) +\frac{1}{2}
W \cdot T \cdot (Q_2+Q_2^T-R_2)\right]\cdot S^T\right\}\,,
\end{align*}
where
\begin{align*}
U:&= \begin{bmatrix}
\Theta & 0  & 0 & 0\\
0 & \Theta & 0  & 0\\
0 &  0 & \frac{\sqrt{1-\Theta^2}}{\sqrt{2}} & 0\\
0 & 0 & 0 & \frac{\sqrt{1-\Theta^2}}{\sqrt{2}}
\end{bmatrix}
\end{align*}
{\tiny
\begin{align*}
 V:&=\begin{bmatrix}
\sqrt{2(1-\Theta^2)}(E^{-1}_{11} +E^{-1}_{12}) & 0  & 0 & 0\\
0 & \sqrt{2(1-\Theta^2)}(E^{-1}_{11} +E^{-1}_{12}) & 0  & 0 \\
0 & 0 &  -\left((\Theta+1)E^{-1}_{11} +(\Theta-1) E^{-1}_{12}\right) & 0\\
0 & 0 & 0 & -\left((\Theta-1)E^{-1}_{11} +(\Theta+1) E^{-1}_{12}\right)
\end{bmatrix}\\
W:&=\begin{bmatrix}
\sqrt{2(1-\Theta^2)}(E^{-1}_{22} +E^{-1}_{12}) & 0  & 0 & 0\\
0 & \sqrt{2(1-\Theta^2)}(E^{-1}_{22} +E^{-1}_{12}) & 0  & 0\\
0 & 0 & -\left((\Theta+1)E^{-1}_{12} +(\Theta-1) E^{-1}_{22}\right) & 0 \\
0 & 0 & 0 & -\left((\Theta-1)E^{-1}_{12} +(\Theta+1) E^{-1}_{22}\right)
\end{bmatrix}
\end{align*}
}
\begin{align*}
    T:&= \begin{bmatrix}
\frac{\sqrt{1-\Theta^2}}{\sqrt{2}} & \frac{\Theta+1}{2}E^{-1}_{11} +\frac{\Theta-1}{2}E^{-1}_{12}  & \frac{\Theta+1}{2}E^{-1}_{12} +\frac{\Theta-1}{2}E^{-1}_{22}\\
\frac{\sqrt{1-\Theta^2}}{\sqrt{2}} & \frac{\Theta-1}{2}E^{-1}_{11} +\frac{\Theta+1}{2}E^{-1}_{12}  & \frac{\Theta-1}{2}E^{-1}_{12} +\frac{\Theta+1}{2}E^{-1}_{22}\\
-2\Theta & \sqrt{2(1-\Theta^2)}(E^{-1}_{11} +E^{-1}_{12}) & \sqrt{2(1-\Theta^2)}(E^{-1}_{12} +E^{-1}_{22})\\
-2\Theta & \sqrt{2(1-\Theta^2)}(E^{-1}_{11} +E^{-1}_{12}) & \sqrt{2(1-\Theta^2)}(E^{-1}_{12} +E^{-1}_{22})
\end{bmatrix}\\
S:&= \begin{bmatrix}
\frac{\sqrt{1-\Theta^2}}{\sqrt{2}} & \frac{\Theta+1}{2}E^{-1}_{11} +\frac{\Theta-1}{2}E^{-1}_{12}  & \frac{\Theta+1}{2}E^{-1}_{12} +\frac{\Theta-1}{2}E^{-1}_{22}\\
\frac{\sqrt{1-\Theta^2}}{\sqrt{2}} & \frac{\Theta-1}{2}E^{-1}_{11} +\frac{\Theta+1}{2}E^{-1}_{12}  & \frac{\Theta-1}{2}E^{-1}_{12} +\frac{\Theta+1}{2}E^{-1}_{22}\\
\frac{\sqrt{1-\Theta^2}}{\sqrt{2}} & \frac{\Theta+1}{2}E^{-1}_{11} +\frac{\Theta-1}{2}E^{-1}_{12}  & \frac{\Theta+1}{2}E^{-1}_{12} +\frac{\Theta-1}{2}E^{-1}_{22}\\
\frac{\sqrt{1-\Theta^2}}{\sqrt{2}} & \frac{\Theta-1}{2}E^{-1}_{11} +\frac{\Theta+1}{2}E^{-1}_{12}  & \frac{\Theta-1}{2}E^{-1}_{12} +\frac{\Theta+1}{2}E^{-1}_{22}
\end{bmatrix}\,.
\end{align*}

\subsection{The term $\frac{c}{2} L_{\vec{n}} g(\vec{\nu}, \vec{\nu})$} Using \eqref{L-n-g-1}, we have
\begin{align*}
& \frac{c}{2} L_{\vec{n}} g(\vec{\nu}, \vec{\nu}) = c\langle \vec{\nu}, \overline\nabla_{\vec{\nu}} \vec{n} \rangle \\
=& \frac{c}{2} \begin{bmatrix}
    - \frac{\sqrt{1-\Theta^2}}{\sqrt{2}}  & - \frac{\sqrt{1-\Theta^2}}{\sqrt{2}} 
\end{bmatrix} E^{-1}\begin{bmatrix}
\widetilde e_1 \\
\widetilde e_2
\end{bmatrix}2(\widetilde  e_1^*, \widetilde e_2^*)E\frac{\p E}{\p r}(\widetilde  e_1^*, \widetilde e_2^*)^T(\widetilde e_1, \widetilde e_2) E^{-1}\begin{bmatrix}
- \frac{\sqrt{1-\Theta^2}}{\sqrt{2}}  \\
- \frac{\sqrt{1-\Theta^2}}{\sqrt{2}} 
\end{bmatrix} \\
=& \frac{c (1-\Theta^2)}{2} (\alpha + 2 \beta + \eta)\\
=& \frac{c (1-\Theta^2)}{1+ (1-\lambda^2(0))\sinh^2(r_0) } \sinh(r_0)\cosh(r_0)(1-\lambda^2(0)) + \frac{bc (1-\Theta^2)}{1+ (1-\lambda^2(0))\sinh^2(r_0) } .
\end{align*}

\subsection{Evolution equation of $\Theta^2$ at the minimum point}
Combining all above, the evolution equation of $\Theta^2$ under the MMCF at the minimum point $(p,r_0)$ is: 
 \begin{align}\label{eq:theta-evolution-modifed-mcf}
&\left(\frac{\partial }{\partial t}  - \Delta\right)\Theta^2- 2|\nabla \Theta|^2 = \left(\frac{\partial }{\partial t}  - \Delta\right)\Theta^2 \\
=& 2(|A|^2-2)\Theta^2+(\overline\nabla_{\vec{\nu}} L_{\vec{n}} g)(e_i, e_i)\Theta \notag\\
&-2(\overline\nabla_{e_i}L_{\vec{n}} g)(\vec{\nu}, e_i)\Theta -2(L_{\vec{n}} g)(e_i, e_j)A_{ij} \Theta +2c\langle \vec{\nu}, \overline\nabla_{\vec{\nu}} \vec{n} \rangle \Theta \notag\\
=& 2\Theta^2\left( H - \frac{-2b(1-\Theta^2)+(1-\lambda^2(0))(1+\Theta^2)\sinh(2r_0)}{\Theta(1+\lambda^2(0) + (1-\lambda^2(0))\cosh(2r_0))}\right)^2 \notag\\
&+\frac{1}{2(1+\lambda^2(0) + (1-\lambda^2(0))\cosh(2r_0))^3}\left(\Gamma_1(1-\Theta^4) +\Gamma_2\Theta^2(1-\Theta^2)\right.\notag\\
& \left.+\Gamma_3\Theta(1-\Theta^2)\sqrt{2(1-\Theta^2)} +\Gamma_4 \Theta^3\sqrt{2(1-\Theta^2)}\right)+\Gamma_5 \Theta (1-\Theta^2)\notag\,,
\end{align}
where (note: $\cosh(4 r) =  8\cosh^2(r)\sinh^2(r) + 1$)
\begin{align*}
\Gamma_1 :=& -4 (1+\lambda^2(0) + (1-\lambda^2(0))\cosh(2r_0))\left((1-\lambda^2(0)) \sinh(2 r_0) -2 b\right)^2,\\
 \Gamma_2 :=& \,4 (1-\lambda^2(0))^2 (1+\lambda^2(0)) \cosh(
     4 r_0) + 16 (1-\lambda^2(0)) ((1-\lambda^2(0))^2+2b^2) \cosh(
     2 r_0)  \\
     &+ 4(1+\lambda^2(0)) (3 a^4 + 3 (-1 + b^2)^2 + 
      2 a^2 (-7 + 3 b^2) - 12 b (1-\lambda^2(0)) \sinh(2 r_0))\\
      &- 
   24 b (1-\lambda^2(0))^2 \sinh(4 r_0),
 \\
 \Gamma_3 :=& -4 ((-3 + a^2 + 2 a b - b^2) m + (3 - a^2 + 2 a b + b^2) n) \cosh(r_0) + 
 4 ((1 + a^2 + 2 a b - b^2) m \\
 &+ (-1 - a^2 + 2 a b + b^2) n) \cosh(
   3 r_0) - 8 (3 a^2 b (m - n) + b (1 + b^2) (-m + n) - a^3 (m + n) \\
   &+ 
    a (1 + 3 b^2) (m + n) + (b (-1 + b^2) (m - n) + 3 a^2 b (-m + n) +
        a^3 (m + n) \\
        &- a (-1 + 3 b^2) (m + n)) \cosh(2 r_0)) \sinh(r_0),\\
\Gamma_4 :=& \,32( (a^2 (m - n) + b^2 (-m + n) + 2 a b (m + n)) \sinh(r_0)\\
& -(b (-m + n) + a (m + n)) \cosh(r_0)) \sinh(2 r_0), \\
  \Gamma_5:= &\frac{2c  \sinh(r_0)\cosh(r_0)(1-\lambda^2(0)) + 2bc}{1+ (1-\lambda^2(0))\sinh^2(r_0) }  = \frac{4c  \sinh(r_0)\cosh(r_0)(1-\lambda^2(0)) + 4bc}{1+\lambda^2(0) + (1-\lambda^2(0))\cosh(2r_0)}
  \end{align*}
and
$$
m= \partial_1 \tilde{A}_{11}(0) = - \partial_1 \tilde{A}_{22}(0), \quad n= \partial_1 \tilde{A}_{12}(0)= \partial_2 \tilde{A}_{11}(0).
$$

\section{Height estimates}\label{height-est}
In this section, we prove a technical height estimate for the evolving surfaces which will be crucial for the proof of the main theorem. Such an estimate ensures the MMCF stays in a compact region under appropriate conditions as long as it exists.

First we derive a general equation for $\Delta u$ in an {\afm}:
\begin{lem}\label{equation-u}
Let $S \subset M^3$ be a closed surface that is a geodesic graph over $\Sigma$, and $u(x)$ is the signed distance of $x \in S$ to $\Sigma$. Then we have:
\begin{equation}\label{eq-u1}
\Delta u = \frac{(1-\lambda^2(x))(\Theta^2+1)\cosh(u)\sinh(u) + b(\Theta^2-1)}{1+(1-\lambda^2(x))\sinh^2(u)} - H\Theta\,,
\end{equation}
where $\Delta$ is the Laplace operator on $S$ with respect to the induced metric and $\pm\lambda(x)$ are the principal curvatures of $\Sigma$ at $x$.
\end{lem}
\begin{proof}
For any point $x \in S$, choose $\{\vec{\nu}, e_1, e_2\}$ as in \eqref{coordinate-change-1} to be a local orthonormal frame of $S$ at $x$. Then at $x$ we can compute
\begin{align*}
  \Delta u  &= \sum_{i=1}^2 \nabla_{e_{i}}\nabla_{e_{i}}u = \sum_{i=1}^2  \nabla_{e_i} \langle \vec{n}, e_i\rangle \notag\\
&= \sum_{i=1}^2  (\langle \overline{\nabla}_{{e}_i} \vec{n}, e_i\rangle +  \langle\vec{n}, \overline{\nabla}_{{e}_i} e_i\rangle)\notag\\
&= \frac{1}{2}\sum_{i=1}^2 L_{\vec{n}} g(e_i, e_i) - H\Theta\\
& = \frac{1}{2}(\Theta^2+1)(\alpha+\eta)+ (\Theta^2-1)\beta- H\Theta\\
& = \frac{(1-\lambda^2(x))(\Theta^2+1)\cosh(u)\sinh(u) + b(\Theta^2-1)}{1+(1-\lambda^2(x))\sinh^2(u)} - H\Theta\,,
\end{align*}
where we have used \eqref{L-n-g-e-i-1} and \eqref{matrixF}.
\end{proof}

\begin{remark}\label{r}
Combining with 
\begin{equation}
    u_t  = - (H-c) \Theta\,,
\end{equation} we have the evolution equation for the hyperbolic distance function $u$ of $S(t)$ along the MMCF:
\begin{align}\label{eq-r}
&\left(\frac{\partial }{\partial t} -\Delta\right)u(x,t) \notag\\
= &- \frac{(1-\lambda^2(x))(\Theta^2+1)\cosh(u)\sinh(u) + b(\Theta^2-1)}{1+(1-\lambda^2(x))\sinh^2(u)}+ c\Theta\,.
\end{align}
\end{remark}

Now we define
$$\mathbf{w}(t):= \min_{(x,t)\in S(t)} u(x,t)$$
and 
$$\mathbf{v}(t):= \max_{(x,t)\in S(t)} u(x,t) $$
to be the minimum and maximum hyperbolic signed distances to $\Sigma$ of the evolving surface $S(t)$ respectively.

\bt[\bf Height estimate for MMCF]\label{height-1}
Let $M^3$ be an almost Fuchsian manifold and $\Sigma = \Sigma(0)$ be the unique closed {\ms} in $M^3$. Assume that
$$\lambda_{max} = \max_{x\in \Sigma} \lambda(x,0) 
\leq \frac{\sqrt{2-c}}{2},$$
and the minimum and maximum hyperbolic signed distances of the initial smooth closed surface $S(0)=S_0 \subset M^3$ to $\Sigma$ are $\mathbf{w}(0) = b_0 >0$ and $\mathbf{v}(0)=a_0>0$ respectively, then under the MMCF \eqref{mmcf} with $c\in [0,2)$ we have for any $t \in [T_1, T_2] \subset [0,T]$ (for $T>0$ as long as the flow exists):
    \begin{align}\label{expdecay2}
\sinh^{-1}\left(e^{-\tau_1 (t-T_1) }\sinh(\mathbf{w}(T_1) -r_c)\right) &+  r_c  \leq u(\cdot,t)  \notag\\
& \leq   \sinh^{-1}\left(e^{-\tau_2 (t-T_1)}\sinh( \mathbf{v}(T_1) - \widetilde r_c )\right) +  \widetilde r_c\,,
\end{align}
where
\begin{equation}
    \tau_1 = \left\{
\begin{aligned}
& 4 \quad \text{if }\quad \mathbf{w}(T_1) \geq r_c\\
&\frac{2-c}{4} \quad \text{if }\quad \mathbf{w}(T_1) \leq r_c
\end{aligned}
\right.
\end{equation}
and
\begin{equation}
    \tau_2 = \left\{
\begin{aligned}
& 4 \quad \text{if }\quad \mathbf{v}(T_1) \leq \widetilde r_c\\
&\frac{2-c}{4} \quad \text{if }\quad \mathbf{w}(T_1) \geq \widetilde r_c.
\end{aligned}
\right.
\end{equation}
Here we designate $r_c:= \tanh^{-1} \left(\frac{c}{2}\right)$, $\widetilde r_c := \tanh^{-1} \left(\frac{c}{2 - (2-c\tanh(\widetilde b_0))\lambda^2_{max}}\right)$ and $\widetilde b_0= \min\{ b_0, r_c\}$. In particular, 
\begin{equation}\label{lower-upper-bound-1}
  \widetilde b_0 \leq  u(x,t) \leq \max\{a_0, \widetilde r_c\}
\end{equation}
for all $t \in [0,T]$ and the flow stays in a compact region.
\et

\bp
Recall that the height function $u(x,t) := \pm\text{dist}((x,t),\Sigma)$ is the signed hyperbolic distance of a point $(x,t)$ on the evolving surface $S(t)$ to the reference {\ms} $\Sigma$. Moreover, under the MMCF it satisfies the evolution equation \eqref{eq-r}.

Now suppose $u(q,t) = \mathbf{w}(t)$. Then at $(q,t)$ we have $\Theta(q,t)=1$ and we can use the Hamilton's trick to get 
\begin{align}
    \mathbf{w}'(t) &= \frac{\partial u}{\partial t}(q,t) \geq \left(\frac{\partial }{\partial t} -\Delta\right) u(q,t) \\
    & = - \frac{2(1-\lambda^2(q,0))\cosh(u(q,t))\sinh(u(q,t))}{1+(1-\lambda^2(q,0))\sinh^2(u(q,t))} + c \notag\\
     & = - \frac{2(1-\lambda^2(q,0))\cosh(\mathbf{w}(t) )\sinh(\mathbf{w}(t) )}{\cosh^2(\mathbf{w}(t)) -\lambda^2(q,0)\sinh^2(\mathbf{w}(t) )}+ c \notag\\
     & = - \frac{2(1-\lambda^2(q,0))\tanh(\mathbf{w}(t) )}{1-\lambda^2(q,0) \tanh^2(\mathbf{w}(t))}+ c \notag\\
     = & - \frac{2(1-\lambda^2(q,0))\tanh(\mathbf{w}(t) ) + c\lambda^2(q,0) \tanh^2(\mathbf{w}(t)) -c }{1-\lambda^2(q,0) \tanh^2(\mathbf{w}(t))}.\notag
\end{align}
Note that 
$$
\mathbf{w}(0) = b_0 > 0\,,
$$
so we can assume that $\mathbf{w}(t)\geq 0$ on $[0,T]$ for some $T>0$ (as long as the flow exists). Using that $r_c = \tanh^{-1}\left(\frac{c}{2}\right)$, on $[0,T]$ we have:
\begin{align*}
(\mathbf{w}(t) - r_c)' & = \mathbf{w}' (t)  \\ 
\geq & - \frac{2(1-\lambda^2(q,0))\tanh(\mathbf{w}(t) ) + c\lambda^2(q,0) \tanh(\mathbf{w}(t)) -c }{1-\lambda^2(q,0) \tanh^2(\mathbf{w}(t))}\\
     = & - \frac{\left(2 - (2-c)\lambda^2(q,0)\right)\left(\tanh(\mathbf{w}(t) ) - \frac{c}{2 - (2-c)\lambda^2(q,0)}\right) }{1-\lambda^2(q,0) \tanh^2(\mathbf{w}(t))}.
\end{align*}
There are two cases that we need to consider. Firstly if $\mathbf{w}(t) \geq r_c$ for $t\in [T_1, T_2] \subset [0,T]$, then on $[T_1, T_2]$ we have:
\begin{align*}
(\mathbf{w}(t) - r_c)' 
\geq & - \frac{2}{1-\lambda_{max}^2}\left(\tanh(\mathbf{w}(t) ) - \frac{c}{2 - (2-c)\lambda^2(q,0)}\right)\\
     \geq & - \frac{2}{1-\lambda_{max}^2}\left(\tanh(\mathbf{w}(t) ) - \tanh(r_c)\right)\\
  = & -\frac{2}{1-\lambda^2_{max}} \tanh(\mathbf{w}(t) - r_c )(1- \tanh(\mathbf{w}(t))\tanh(r_c ) )\\
  \geq & -4 \tanh(\mathbf{w}(t) - r_c ),
\end{align*}
where we used that $\lambda_{max}^2 \le \frac{2-c}{4} \le \frac12$. Therefore 
\begin{align*}
(\mathbf{w}(t) - r_c)' & \geq -4 \tanh(\mathbf{w}(t) - r_c )
\end{align*}
and thus on $[T_1,T_2]$
  \begin{equation}
    \left(e^{4t}\sinh (\mathbf{w}(t) - r_c )\right)'\geq 0.
\end{equation}
Then we have on $[T_1,T_2]$
\begin{align*}
    u(x,t) \geq \mathbf{w}(t) &\geq \sinh^{-1}\left(e^{-4(t-T_1) }\sinh(\mathbf{w}(T_1) -r_c)\right) + r_c.
\end{align*}
This in fact yields that 
\begin{equation}
    u(x,t)\geq \mathbf{w}(t) \geq r_c,
\end{equation}
for all $t\geq T_1$ as long as the flow exists.

For the other case, if $\mathbf{w}(t) \leq r_c$ for $t\in [T_1, T_2] \subset [0,T]$, then we use again that $\lambda_{max}^2 \leq \frac{2-c}{4}\leq \frac12$ on $[T_1, T_2]$ to estimate:
\begin{align*}
(\mathbf{w}(t) - r_c)' 
\geq & - \frac{2-(2-c)\lambda_{max}^2}{2}\left(\tanh(\mathbf{w}(t) ) - \frac{c}{2 - (2-c)\lambda^2(q,0)}\right)\\
     \geq & -\frac{2-(2-c)\lambda_{max}^2}{2}\left(\tanh(\mathbf{w}(t) ) - \tanh(r_c)\right)\\
  \geq & -\frac{1}{2}\tanh(\mathbf{w}(t) - r_c )(1- \tanh(\mathbf{w}(t))\tanh(r_c ) )\\
  \geq & -\frac{1}{2}\tanh(\mathbf{w}(t) - r_c )(1-\tanh(r_c ) )\\
  = & -\frac{2-c}{4} \tanh(\mathbf{w}(t) - r_c ).
\end{align*}
Then similarly on $[T_1,T_2]$,
\begin{align*}
    u(x,t) \geq \mathbf{w}(t) 
    &\geq \sinh^{-1}\left(e^{-\frac{2-c}{4} (t-T_1) }\sinh(\mathbf{w}(T_1) -r_c)\right) + r_c\\
    =& r_c - \sinh^{-1}\left(e^{-\frac{2-c}{4} (t-T_1) }\sinh( r_c -\mathbf{w}(T_1))\right) \geq  \mathbf{w}(T_1).
\end{align*}
In particular, combining the above two cases we have
\begin{equation}\label{keeplower-1}
    u(x,t)\geq \mathbf{w}(t)  \geq \min\{b_0, r_c\} := \widetilde b_0 \geq 0,
\end{equation}
for all $t\in [0,T]$.

Now we estimate the upper bound for $u(p,t)$. Suppose $u(p,t) = \mathbf{v}(t)$. Note that we have \eqref{keeplower-1} as long as the flow exists and also $\Theta = 1$ at the maximum point. Then similarly we can estimate:
\begin{align*}
    \mathbf{v}'(t) = & \frac{\partial u}{\partial t}(p,t) \leq \left(\frac{\partial }{\partial t} -\Delta\right) u(p,t)\\
     = & - \frac{2(1-\lambda^2(p,0))\tanh(\mathbf{v}(t) )}{1-\lambda^2(p,0) \tanh^2(\mathbf{v}(t))}+ c\\
     \leq & - \frac{2(1-\lambda^2(p,0))\tanh(\mathbf{v}(t) ) + c\lambda^2(p,0) \tanh(\widetilde b_0)\tanh(\mathbf{v}(t)) -c }{1-\lambda^2(p,0) \tanh^2(\mathbf{v}(t))}\\
     = & - \frac{\left(2 - (2-c\tanh(\widetilde b_0))\lambda^2(p,0)\right)\left(\tanh(\mathbf{v}(t) ) - \frac{c}{2 - (2-c\tanh(\widetilde b_0))\lambda^2(p,0)}\right) }{1-\lambda^2(p,0) \tanh^2(\mathbf{v}(t))}.
\end{align*}

We also work in two cases. Firstly, for $\widetilde r_c := \tanh^{-1} \left(\frac{c}{2 - (2-c\tanh(\widetilde b_0))\lambda^2_{max}}\right)$, if $\mathbf{v}(t) \geq \widetilde r_c$ for $t\in [T_1, T_2] \subset [0,T]$, then on $[T_1, T_2]$ we have
\begin{align*}
\mathbf{v}'(t) \leq &  -\left(2-(2-c\tanh(\widetilde b_0))\lambda^2_{max}\right)\left(\tanh(\mathbf{v}(t) ) -  \tanh(\widetilde r_c)\right)\\
     \leq &- \left(2 - (2-c\tanh(\widetilde b_0))\lambda^2_{max}\right) \left( 1-  \tanh(\widetilde r_c)\right) \left(\tanh(\mathbf{v}(t)-\widetilde r_c )\right)\\
     =& -(2-c -(2-c\tanh(\widetilde b_0))\lambda^2_{max} )\left(\tanh(\mathbf{v}(t)-\widetilde r_c )\right)\\
     \leq & -(2-c -2\lambda_{max}^2 )\tanh(\mathbf{v}(t)-\widetilde r_c ) \\
     \leq & -\frac{2-c}{4}\tanh(\mathbf{v}(t)-\widetilde r_c ),
\end{align*}
where we used $\lambda_{max}^2 \le \frac{2-c}{4}$. Therefore
\begin{align*}
(\mathbf{v}(t) - \widetilde r_c)' & \leq -\frac{2-c}{4} \tanh(\mathbf{v}(t) - \widetilde r_c )
\end{align*}
and thus on $[T_1,T_2]$
  \begin{equation}
    \left(e^{\frac{2-c}{4} t}\sinh (\mathbf{v}(t) - \widetilde r_c )\right)'\leq 0.
\end{equation}
Then we have on $[T_1,T_2]$
\begin{align*}
    u(x,t) \leq \mathbf{v}(t) &\leq \sinh^{-1}\left(e^{-\frac{2-c}{4}(t-T_1) }\sinh(\mathbf{v}(T_1) - \widetilde r_c)\right) + \widetilde r_c.
\end{align*}

For the other case, if $\mathbf{v}(t) \leq \widetilde r_c$ for $t\in [T_1, T_2] \subset [0,T]$, then on $[T_1, T_2]$ we have:
\begin{align*}
    \mathbf{v}'(t) \leq &  - \frac{2 - (2-c\tanh(\widetilde b_0))\lambda^2_{max}}{1-\lambda_{max}^2}\left(\tanh(\mathbf{v}(t) ) -  \tanh(\widetilde r_c)\right)\\
     \leq &- \frac{2 - (2-c\tanh(\widetilde b_0))\lambda^2_{max}}{1-\lambda_{max}^2}\left( 1- \tanh(\widetilde r_c)\right) \left(\tanh(\mathbf{v}(t)-\widetilde r_c )\right)\\
          = &- \frac{2 - (2-c\tanh(\widetilde b_0))\lambda^2_{max}-c}{1-\lambda_{max}^2}\left(\tanh(\mathbf{v}(t)-\widetilde r_c )\right)\\
     \leq &- \frac{2}{1-\lambda_{max}^2} \tanh(\mathbf{v}(t)-\widetilde r_c )\\
     \leq &- 4 \left(\tanh(\mathbf{v}(t)-\widetilde r_c )\right),
\end{align*}
where we used $\lambda_{max}^2 \leq \frac12$ and $(2-c\tanh(\widetilde b_0))\lambda^2_{max} \ge 0$.
Here we have also used the facts that
$$
\frac{2-(2-c\tanh(\widetilde b_0))\lambda^2}{1-\lambda^2}
$$
is non-decreasing with respect to $\lambda>0$ and when $\lambda_{max}\le \frac{\sqrt{2-c}}{2}$ we have
\begin{equation}
    \frac{c}{2 - (2-c\tanh(\widetilde b_0))\lambda^2_{max}} <1.
\end{equation}
Thus 
 \begin{equation}
    \left(e^{4 t}\sinh (\mathbf{v}(t) - \widetilde r_c )\right)'\leq 0
\end{equation}
and therefore on $[T_1, T_2]$ we have
\begin{align*}
    u(x,t)\leq \mathbf{v}(t) &\leq \sinh^{-1}\left(e^{-4(t-T_1)}\sinh(\mathbf{v}(T_1)- \widetilde r_c)\right) + \widetilde r_c\\
    &= \widetilde r_c - \sinh^{-1}\left(e^{-4(t-T_1)}\sinh(\widetilde r_c - \mathbf{v}(T_1))\right)\\
    &\leq \widetilde r_c.
\end{align*}
This also yields that if $\mathbf{v}(0) = a_0\leq \widetilde r_c$ then $\mathbf{v}(t) \leq \widetilde r_c$ for all $t>0$ as long as the flow exists. This completes the proof.
\end{proof}
Similar to the mean convex MCF in Euclidean space, we prove a monotone property for the MMCF in our setting:
\begin{lem} \label{mean-convex-1} Let $M^3$ be an {\afm} and $\Sigma = \Sigma(0)$ be the unique closed {\ms} in $M^3$. Assume that $H_{min}(0) = \min_{S_0} H \geq c$, then under the MMCF \eqref{mmcf} with $c\in [0,2)$, the maximum height $\mathbf{v}(t)$ of the evolving surface is non-increasing for all $t >0$ as long as the flow exists.
\end{lem}

\begin{proof} We first show that $H>c$ as long as the flow exists. Consider the function 
$f( \cdot,t)$ on the evolving surface $S(t)$: 
\begin{equation}
f(\cdot,t) = e^{2t}(H(\cdot) - c),
\end{equation}
with $f(\cdot,0) \geq 0$ and we let $f_{min}(t) = \min_{S(t)} f(\cdot, t)$. We calculate the evolution equation for $f$, using also the evolution equation of $H$:
\begin{align*}
f_t - \Delta f &= 2f - e^{2t}(H_t-\Delta H)\\
&= 2f-e^{2t}(H-c)(|A|^2-2)\\
&=|A|^2 f.
\end{align*}
Now it follows from Hamilton's trick that, assuming $f(p,t) = f_{min}(t)\geq 0$ on $[0,T]$ for some $T>0$ as long as the flow exists, we have
\begin{align*}
    \frac{d}{dt} f_{min} = \frac{\partial }{\partial t} f(p,t) \geq \left(\frac{\partial }{\partial t} - \Delta\right) f (p,t) = |A|^2 f(p,t) \geq 0
\end{align*}
and thus for any $t\in [0,T]$
\begin{align}\label{H-and-c-1}
     H(\cdot, t) - c \geq  e^{-2 t}(H_{min}(0) -c )\geq 0\,.
\end{align}

Recall that $$\mathbf{v}(t):= \max_{(x,t)\in S(t)} u(x,t) $$ and suppose $u(p,t) = \mathbf{v}(t)$. We can use the fact that $\Theta(p,t)=1$
\begin{align}
    \mathbf{v}'(t) & = \frac{\partial u(p,t)}{\partial t} = c-H(p,t) \leq e^{-2t}(c-H_{min}(0))  \leq  0.
\end{align}
Therefore, as long as the flow exists $\mathbf{v}(t)$ is non-increasing and we have
\begin{equation}\label{height-upper-1}
    \mathbf{v}(t)\leq \mathbf{v}(0) .
\end{equation}
\end{proof}
\begin{remark}\label{r-slice-1}
By Remark \ref{mean-curvature-99}, the mean curvature of the equidistant surface $\Sigma(r)$ at $(p,r)$ is equal to
$$
H(p,r) =  \frac{2(1-\lambda^2(p,0)) \tanh(r)}{1- \lambda^2(p,0) \tanh^2(r)} 
$$
with
$$
\lim_{r\to +\infty} H(p,r) = 2, \quad \forall \,p \in \Sigma.
$$
Therefore, we can always find a such initial surface $S_0$ as in the Lemma \ref{mean-convex-1}. In fact, when
\begin{equation}
r \geq \tanh^{-1}\left(\frac{c}{1- \lambda_{max}^2 + \sqrt{(1-\lambda_{max}^2)^2 + c^2\lambda_{max}^2}}\right):= \hat{r}_c
\end{equation}
then the mean curvature of $\Sigma(r)$ is at least $c\in [0,2)$. Moreover, suppose $b_0 \geq r_c = \tanh^{-1}\left(\frac{c}{2}\right)$ in Theorem \ref{height-1} then $\widetilde b_0 = r_c$ in \eqref{keeplower-1} and we have
\begin{align}
   & \tanh(\hat{r}_c) - \tanh(\widetilde r_c)\notag \\
   = &  \frac{c}{1- \lambda_{max}^2 + \sqrt{(1-\lambda_{max}^2)^2 + c^2\lambda_{max}^2}}- \frac{c}{2 - (2-c\tanh(\widetilde b_0))\lambda^2_{max}}\notag\\
    = & \frac{c}{1- \lambda_{max}^2 + \sqrt{(1-\lambda_{max}^2)^2 + c^2\lambda_{max}^2}}- \frac{c}{2 - (2-\frac{c^2}{2})\lambda^2_{max}} \leq  0\notag
\end{align}
for all $c\in [0,2)$ and all $\lambda_{max}\in [0,1)$.

We also have
\begin{align*}
   & \tanh(\hat{r}_c) - \tanh(r_c)\notag \\
   = &  \frac{c}{1- \lambda_{max}^2 + \sqrt{(1-\lambda_{max}^2)^2 + c^2\lambda_{max}^2}}- \frac{c}{2} \geq  0
\end{align*}
for all $c \in [0,2)$ and all $\lambda_{max}\in [0,1)$.
\end{remark}

\section{Proofs of main results}\label{proof-of-main}
In this section, we assemble our previous estimates and prove main results of this paper, namely Theorem ~\ref{main} (which is detailed analysis of the long term behavior of the MMCF in a subclass of {\afm}s), and Theorem \ref{main-thm-2} (which confirms the Conjecture \ref{Thurston-Conj} for this subclass of {\afm}s). 
\subsection{Further estimates for $\Gamma_i$'s} 
From the explicit expressions for $\Gamma_i$'s in \eqref{eq:theta-evolution-modifed-mcf}, if $|A_{\Sigma}|_{C^1}$ is sufficiently small on $\Sigma$, namely, $\lambda_0, |a|,|b|,|m|,|n| \in [0,\epsilon]$ for some $\epsilon \in (0,1)$ sufficiently small to be determined later, we have the upper bound estimates for $|\Gamma_i|$'s:
\be\label{gamma-upper-1}
|\Gamma_1| \leq  4(\cosh(2r_0)+2)(\sinh(2r_0)+2)^2 
\leq 108\cosh^3(2r_0),
\ene
and 
\begin{align}
|\Gamma_2| &\leq 4\left( 12 \cosh(2r_0) + 2 \cosh(4r_0) + 40 + 24\sinh(2r_0)\ + 6\sinh(4r_0)\right) \notag\\
&\leq 368\cosh^2(2r_0), \label{gamma-upper-2}
\end{align}
and 
\begin{align}
|\Gamma_3| & \leq 40 \epsilon \cosh(r_0)+ 40 \epsilon \cosh(3r_0) + (160\epsilon+ 96\epsilon \cosh(2r_0))\sinh(r_0) \notag\\
&\leq 376\cosh^2(2r_0),\label{gamma-upper-3}
\end{align}
and
\begin{align}
|\Gamma_4| & \leq 32\left( 8\epsilon^3\sinh(r_0) + 4\epsilon^2\cosh(r_0)\right)\sinh(2r_0)\notag\\
&\leq 384\cosh^2(2r_0),\label{gamma-upper-4}
\end{align}
and lastly 
\be\label{gamma-upper-5}
|\Gamma_5| \leq 2\tanh(2r_0) + 4 \leq 6.
 \ene
Furthermore, we inspect more closely some terms in \eqref{eq:theta-evolution-modifed-mcf} as follows: Since we need to estimate $\Gamma_1(1-\Theta^4)$, we notice that 
\begin{align*}
&\Gamma_1 = -4 (1+\lambda^2(0) + (1-\lambda^2(0))\cosh(2r_0))\left( (1-\lambda^2(0))^2\sinh^{2}(2r_0) \right)\\
& +(16b(1-\lambda^2(0))\sinh (2r_0)-16 b^2)(1+\lambda^2(0) + (1-\lambda^2(0))\cosh(2r_0)) \,,
\end{align*}
then we have 
\begin{eqnarray*}
&16(b(1-\lambda^2(0))\sinh(2r_0)-b^2)(1+\lambda^2(0)+(1-\lambda^2(0))\cosh(2r_0))(1-\Theta^4)\\
&\geq  -32\epsilon(\sinh(2r_0) + \epsilon)(\cosh(2r_0)+2)(1-\Theta^2)\\
&\geq  -96 \epsilon \left(\frac{1}{2} \sinh(4r_0) + \cosh(2r_0)\right)(1-\Theta^2)\\
&\geq -144 \epsilon \cosh(4r_0)\Theta^2(1-\Theta^2)\frac{1}{\Theta^2}\,.
\end{eqnarray*}
Moreover, we have 
\begin{align*}
&\Gamma_2 \geq 4(1-\epsilon^2)^2\cosh(4r_0)+ 16(1-\lambda(0)^2)^3\cosh(2r_0) -4(1+\epsilon^2)(14\epsilon^2\\
& \ \ \ +12\epsilon\sinh(2r_0))-24\epsilon \sinh(4r_0)\notag\\
\geq & 4(1-\epsilon^2)^2\cosh(4r_0) - 120\epsilon \sinh(4r_0)-72\epsilon^2 + 16(1-\lambda(0)^2)^3\cosh^2 r_0 \notag\\
\geq & \left(4(1-\epsilon^2)^2 - 192\epsilon\right)\cosh(4r_0) + 16(1-\lambda(0)^2)^3\cosh^2 r_0\,,\notag
\end{align*}
and 
\begin{align*}
&\Gamma_3\Theta(1-\Theta^2)\sqrt{2(1-\Theta^2)} \geq  \left(-8\epsilon(3+4\epsilon^2)\cosh r_0 - 8\epsilon(1+4\epsilon^2)\cosh(3r_0)\right.\\
&\left.-32\epsilon^2(1+4\epsilon^2)(1+\cosh(2r_0))\sinh r_0\right)\cdot \Theta(1-\Theta^2)\sqrt{2(1-\Theta^2)}\\
&\geq  -112\epsilon \cosh(3r_0)-160 \epsilon^2\sinh(3r_0)\Theta(1-\Theta^2)\sqrt{2(1-\Theta^2)} \\
&\geq -272\sqrt{2}\epsilon \cosh(3r_0) \Theta^2(1-\Theta^2)\frac{1}{\Theta\sqrt{1-\Theta^2}}\\
&\geq -272\sqrt{2}\epsilon \cosh(4r_0) \Theta^2(1-\Theta^2)\frac{1}{\Theta\sqrt{1-\Theta^2}}\,,
\end{align*}
and
\begin{align*}
    & \Gamma_4 \Theta^3\sqrt{2(1-\Theta^2)}
  \geq  -384 \sqrt{2}\epsilon^2 \cosh r_0\sinh(2r_0)\Theta^2(1-\Theta^2)\frac{1}{\sqrt{1-\Theta^2}}\\
  \geq & -384 \sqrt{2}\epsilon^2 \sinh(3r_0)\Theta^2(1-\Theta^2)\frac{1}{\sqrt{1-\Theta^2}}\\
  \geq & -384 \sqrt{2}\epsilon^2 \cosh(4r_0)\Theta^2(1-\Theta^2)\frac{1}{\sqrt{1-\Theta^2}}\,,
  \end{align*}
and
  \begin{align*}
  \frac{2bc\Theta (1-\Theta^2)}{1+ (1-\lambda^2(0))\sinh^2 r_0 } =&\frac{ 8bc(1+\lambda^2(0)+(1-\lambda^2(0)) \cosh(2r_0))^2\Theta (1-\Theta^2)}{2(1+\lambda^2(0)+(1-\lambda^2(0)) \cosh(2r_0))^3}\\
   \geq & \frac{-16 \epsilon (2+\cosh(2r_0))^2 \Theta^2 (1-\Theta^2)\frac{1}{\Theta}}{2(1+\lambda^2(0)+(1-\lambda^2(0)) \cosh(2r_0))^3} \\
   \geq & \frac{-112\epsilon\Theta^2(1-\Theta^2)\frac{1}{\Theta}}{2(1+\lambda^2(0)+(1-\lambda^2(0))\cosh(2r_0))^3}.
\end{align*}
\subsection{Positive lower bound for the angle} The heart of matter is to show the evolving surfaces stay as graphs of the fixed {\ms}, this amounts to prevent $\Theta$ from going to zero along the flow. Now we prove the following key lemma establishing a positive lower bound for $\Theta$. 
\bl\label{keylemma}
Let $M^3$ be an {\afm} and $\Sigma =\Sigma(0)$ be the unique closed {\ms} with {\pc}s $\pm \lambda(x,0) \in (-1,1)$ in $M^3$. There exists a universal constant $\epsilon$ as in \eqref{unifconst}: 
$$
    \epsilon \in \left[ 0, 1-\frac{1}{\left(1+\frac{\epsilon_1}{16}\right)^2}\right) \cap \left[0, \frac{\sqrt{\epsilon_1}}{12}\right]\cap \left[0, 7\times 10^{-6} \right],
$$
where $\epsilon_1>0$ is from \eqref{eps_1}, such that if the {\sff} of $\Sigma$ satisfies 
$$
\|A_{\Sigma}\|_{C^1(\Sigma)} \leq \epsilon \,,
$$
then the angle function $\Theta(\cdot, t)$ for the MMCF \eqref{mmcf} with $c\in \left[0, 2-\frac{\epsilon_1}{2}\right]$ starting from any equidistant surface $S_0 = \Sigma(r) \subset M^3$ with \eqref{r-range-1} has a positive lower bound, namely,
\be\label{theta-bd}
\Theta \geq \frac{1}{1+\frac{\epsilon_1}{8}}.
\ene
\el
\begin{proof}
Let $0<\epsilon_1 \ll 1$ be given by \eqref{eps_1} so that the two ends of $M^3$ are monotonically foliated by closed CMC surfaces of mean curvature ranging within $(-2, -2+\epsilon_1) \cup (2-\epsilon_1,2)$. For any
$$ c \in \left[0, 2-\frac{\epsilon_1}{2}\right],$$
by Remark \ref{r-slice-1}, assuming that 
\begin{equation}\label{lambda-range}
    \lambda_{max} \in \left[0, \frac{\sqrt{\epsilon_1}}{12}\right] \subset \left[0, \frac{\sqrt{2-c}}{2}\right] \cap  \left[0, \frac{1}{2}\right]
\end{equation}
so that if $S_0 = \Sigma(r)$ is the $r$-equidistant surface with 
$$r \geq \tanh^{-1} \left(\frac{c}{2 - (2-\frac{c^2}{2})\lambda^2_{max}}\right) = \widetilde r_c,$$ then the mean curvature of $S_0$ is at least $c \in [0,2)$, with $\min_{S_0}\Theta^2 =1$ and $a_0=b_0 \geq \widetilde r_c > r_c = \tanh^{-1}\left(\frac{c}{2}\right)$. Note also that \eqref{lambda-range} ensures that both Theorem \ref{height-1} (with $\widetilde b_0 = r_c$) and Lemma \ref{mean-convex-1} apply in our setting.

Let $\epsilon>0$ be chosen so that it satisfies \eqref{r-range-1}, which implies that 
$$
\frac{1}{1+\frac{\epsilon_1}{16}} < \sqrt{1-\epsilon},
$$
and we use the condition \eqref{r-range-1} to find:
\begin{equation}
    \epsilon \leq \frac{1- \left(\frac{1}{1+\frac{\epsilon_1}{16}}\right)^2}{e^{8  \cosh (2a_0)}}.
\end{equation} 
Now letting $$\phi(t) = \min_{S(t)} \Theta^2 \in (0,1]$$ on $[0,T]$ for some $T>0$ as long as the flow exists and stays graphic over $\Sigma$. Since $\phi(0) = \min_{S_0}\Theta^2 =1$, the MMCF exists for some time $T>0$ and $\phi(t)$ possibly decays along the flow. 
Let $T_1 \in (0, \infty]$ be the first time $t>0$ such that $\phi(T_1) = 1-\epsilon$. If this never happens then we set $T_1 = \infty$ and we are done. So we may assume $T_1 < \infty$.

We first note that there exists 
\begin{equation}
\epsilon_0 = 7\times 10^{-6}
\end{equation}
such that if $\epsilon\leq \epsilon_0$ and
\begin{equation}\label{theta-range-10}
\Theta \in \left[\frac{1}{2}, \sqrt{1 - \epsilon}\right]\,,
\end{equation}
then
\begin{align}
    2(1-\epsilon^2)^2 - \left(192+ \frac{144}{\Theta^2}+\frac{272\sqrt{2}}{\Theta\sqrt{1-\Theta^2}}+ \frac{384\sqrt{2}\epsilon}{\sqrt{1-\Theta^2}}+\frac{112}{\Theta} \right)\epsilon \geq 0.
\end{align}
To see this, we use that if $\epsilon$ and $\Theta$ are in those ranges, then
$$
 \frac{1}{\Theta^2} \leq 4, \quad \frac{1}{\Theta\sqrt{1-\Theta^2}}\leq \frac{2}{\sqrt{\epsilon}}, \quad \frac{\epsilon}{\sqrt{1-\Theta^2}}\leq \sqrt{\epsilon} \quad \text{and} \quad \frac{1}{\Theta} \leq 2.
$$
Therefore if
\begin{equation}\label{theta-range-4}
\|A_{\Sigma}\|_{C^1} \leq \epsilon \leq 7\times 10^{-6} \quad \text{and}\quad \Theta \in \left[\frac{1}{2}, \sqrt{1 - \epsilon}\right]\,,
\end{equation}
then half of the first (positive) term of $\Gamma_2 \Theta^2(1-\Theta^2)$ in \eqref{eq:theta-evolution-modifed-mcf} will dominate the other potentially negative terms, namely,
\begin{align*}
&\Gamma_2 \Theta^2(1-\Theta^2) +(16b(1-\lambda^2(0))\sinh (2r)-16 b^2)(1+\lambda^2(0) + (1-\lambda^2(0))\cosh(2r)) \\
&\cdot (1-\Theta^4) +\Gamma_3\Theta(1-\Theta^2)\sqrt{2(1-\Theta^2)} +\Gamma_4 \Theta^3\sqrt{2(1-\Theta^2)}\\
&+ 8bc(1+\lambda^2(0) + (1-\lambda^2(0))\cosh(2r))^2\Theta (1-\Theta^2) \\
\geq & 2 (1-\lambda^2(0))^2 \cosh(
     4 r)\Theta^2(1-\Theta^2) \geq 16 (1-\lambda^2(0))^2 \cosh^2(
     r)\sinh^2(
     r)\Theta^2(1-\Theta^2)\,.
\end{align*}
In this case, using the height estimate in Theorem \ref{height-1}:
\begin{equation}
r_c \leq u(\cdot,t) \leq   \sinh^{-1}\left(e^{-\frac{2-c}{4} t}\sinh( a_0 - \widetilde r_c )\right) +  \widetilde r_c
\end{equation}
and the fact that
\begin{align*}
2(1-\lambda^2(0))\cosh^2(r_0)\leq 1+\lambda^2(0) + (1-\lambda^2(0))\cosh(2r_0)\leq 2(1+\lambda^2(0))\cosh^2(r_0)\,,
\end{align*}
at the minimum point $(p,r_0)$ we have
\begin{align}
    \left(\frac{\partial }{\partial t}  - \Delta\right)\Theta^2 
    \geq  &\frac{-8(1-\lambda^2(p,0))^2\cosh^2(r_0)\sinh^{2}(r_0)}{(1+\lambda^2(p,0) + (1-\lambda^2(p,0))\cosh(2r_0))^2} (1+\Theta^2) (1-\Theta^2) \notag\\
    &+ \frac{ 8(1-\lambda^2(p,0))^2\cosh^2(r_0)\sinh^2(r_0)}{(1+\lambda^2(p,0) + (1-\lambda^2(p,0))\cosh(2r_0))^3}\Theta^2(1-\Theta^2)\notag\\
    & + \frac{4c  \sinh(r_0)\cosh(r_0)(1-\lambda^2(p,0)) }{1+\lambda^2(p,0) + (1-\lambda^2(p,0))\cosh(2r_0)}\Theta(1-\Theta^2)\notag\\
    \geq  & - \frac{4\cosh(2\widetilde r_c)\sinh^2(a_0- \widetilde r_c)}{\cosh^2(r_c)}  e^{-\frac{2-c}{2}t} (1-\Theta^2) \notag\\
     &-\frac{2\sinh(2(a_0-\widetilde r_c))\sinh(2 \widetilde r_c)}{\cosh^2(r_c)} e^{-\frac{2-c}{4}t} (1-\Theta^2)\notag\\
    & -\frac{8(1-\lambda^2(p,0))^2\cosh^2(r_0)\sinh^{2}(\widetilde r_c)}{(1+\lambda^2(p,0) + (1-\lambda^2(p,0))\cosh(2r_0))^2} (1+\Theta^2) (1-\Theta^2) \label{term-1}\\
    &+ \frac{ 8(1-\lambda^2(p,0))^2\cosh^2(r_0)\sinh^2(r_0)}{(1+\lambda^2(p,0) + (1-\lambda^2(p,0))\cosh(2r_0))^3}\Theta^2(1-\Theta^2) \label{term-2}\\
    & + \frac{4c  \sinh(r_0)\cosh(r_0)(1-\lambda^2(p,0)) }{1+\lambda^2(p,0) + (1-\lambda^2(p,0))\cosh(2r_0)}\Theta(1-\Theta^2). \label{term-3}
\end{align}
Our next goal is to choose appropriate range of $\min_{S(t)}\Theta = \Theta(p,r_0)$ so that the sum of the last three terms \eqref{term-1}, \eqref{term-2} and \eqref{term-3} is non-negative, so that the following differential inequality holds: 
\begin{align}\label{diff-ineq-for-theta-1}
    \frac{d }{d t} \phi(t)
\geq  &-\frac{4\cosh(2\widetilde r_c)\sinh^2(a_0- \widetilde r_c) + 2\sinh(2(a_0-\widetilde r_c))\sinh(2 \widetilde r_c)}{\cosh^2(r_c)} \notag\\
    &\cdot e^{-\frac{2-c}{4}t} (1-\phi(t)) .
\end{align}
To do so, at $(p,r_0)$ we need to have (using that $4-c^2 = (2+c)(2-c)\geq 2\cdot \frac{\epsilon_1}{2} = \epsilon_1$ and $\lambda_{max}^2\leq \frac{1}{4}$)
\begin{align*}
   &\left( 1- \frac{ \sinh^2(r_0)}{(1+\lambda^2(p,0) + (1-\lambda^2(p,0))\cosh(2r_0))\sinh^2(\widetilde r_c)} \right) \Theta^2 \\
   & - \frac{c\tanh(r_0)(1+\lambda^2(p,0) + (1-\lambda^2(p,0))\cosh(2r_0))}{2(1-\lambda^2(p,0))\sinh^2(\widetilde r_c)}\Theta + 1\\
   \leq & \left( 1- \frac{ \sinh^2(r_c)}{2(1+\lambda^2_{max})\cosh^2(r_c)\sinh^2(\widetilde r_c)} \right) \Theta^2  - \frac{c\sinh(r_c)\cosh(r_c)}{\sinh^2(\widetilde r_c)}\Theta + 1\\
   =& \left(1 - \frac{(4-c^2)(4-8\lambda_{max}^2+(4-c^2)\lambda_{max}^4}{32(1+\lambda_{max}^2)}\right)\Theta^2 - \left(2 - 4\lambda_{max}^2 + \frac{4-c^2}{2}\lambda_{max}^4 \right)\Theta + 1 \\
   \leq & \left(1-\frac{4\epsilon_1(1-2\lambda_{max}^2)}{64}\right)\Theta^2-2\left(1-2\lambda_{max}^2\right)\Theta+1 \leq 0.
   \end{align*}
Note that by our choice \eqref{lambda-range} we have
\begin{equation}\label{lambda-range-3}
    \lambda_{max}^2\leq \frac{\epsilon_1}{12^2}\leq \frac{\epsilon_1}{4(32+\epsilon_1)}
\end{equation}
and so that
$$
\epsilon_1-2(32+\epsilon_1)\lambda_{max}^2\geq \frac{\epsilon_1}{2},
$$
then the solution to the last inequality before \eqref{lambda-range-3} is
\begin{align}\label{theta-range-3}
   &\frac{4}{4-8\lambda_{max}^2+\sqrt{\epsilon_1-2(32+\epsilon_1)\lambda_{max}^2+64\lambda_{max}^4}}\\
   \leq &\Theta \leq  \frac{4}{4-8\lambda_{max}^2-\sqrt{\epsilon_1-2(32+\epsilon_1)\lambda_{max}^2+64\lambda_{max}^4}}.\notag
\end{align}
Therefore our choice 
\begin{equation}
   \frac{1}{1+\frac{\epsilon_1}{8}} \leq \Theta \leq 1
\end{equation} 
is within the range of the solution \eqref{theta-range-3}. Combining this with \eqref{theta-range-10} and \eqref{lambda-range}, we can choose any
\begin{equation}\label{eps-fixed}
    \epsilon \in \left[ 0, 1- \frac{1}{\left(1+\frac{\epsilon_1}{16}\right)^2}\right) \cap \left[0, \frac{\sqrt{\epsilon_1}}{12}\right]\cap \left[0, 7\times 10^{-6} \right]
\end{equation}
so that
$$
\frac12< \frac{1}{1+\frac{\epsilon_1}{8}}<\frac{1}{1+\frac{\epsilon_1}{16}} < \sqrt{1-\epsilon}.
$$
Then as long as 
\begin{equation}\label{theta-range-5}
\|A_{\Sigma}\|_{C^1} \leq \epsilon \quad \text{and}\quad \min_{S(t)}\Theta \in \left[\frac{1}{1+\frac{\epsilon_1}{8}},\sqrt{1 - \epsilon}\right]
\end{equation}
we have the differential inequality \eqref{diff-ineq-for-theta-1} for $\phi(t)$.

Suppose the lower bound \eqref{theta-bd} does not hold, then we let $T_2> T_1$ be the first time that $\phi(T_2) =\frac{1}{(1+\frac{\epsilon_1}{8})^2}$. Then on $[T_1, T_2]$, \eqref{theta-range-5} holds, and we can apply the differential inequality \eqref{diff-ineq-for-theta-1},  
that is,
 \begin{align*}
&\frac{d }{d t} \phi(t)
\geq  -C_1 e^{-C_2t}(1-\phi(t)),
\end{align*}
where
 \begin{align*}
C_1 & = \frac{4\cosh(2\widetilde r_c)\sinh^2(a_0- \widetilde r_c) + 2\sinh(2(a_0-\widetilde r_c))\sinh(2 \widetilde r_c)}{\cosh^2(r_c)}, \\
C_2 & = \frac{2-c}{4}\geq \frac{\epsilon_1}{8}.
\end{align*}
Then
 \begin{align*}
  \frac{\frac{d }{d t} \phi(t)}{1-\phi(t)}\geq -C_1 e^{-C_2t}
\end{align*}
and
 \begin{align*}
 \frac{d }{d t} \ln(1-\phi(t))\leq C_1 e^{-C_2t}.
\end{align*}
Integrating over $[T_1,t]$ we get
 \begin{align*}
 \ln(1-\phi(t)) - \ln(1-\phi(T_1)) &\leq C_1 \int_{T_1}^t e^{-C_2t} dt\\
 & = \frac{C_1}{C_2}(e^{-C_2T_1}-e^{-C_2t} )\leq \frac{C_1}{C_2}.
\end{align*}
Thus,
 \begin{align*}
 \ln(1-\phi(t)) & \leq  \ln(1-\phi(T_1)) + \frac{C_1}{C_2} \\
 & \leq  \ln(1-\phi(T_1)) + 
 \frac{16\cosh(2\widetilde r_c)\sinh^2(a_0- \widetilde r_c) + 2\sinh(2(a_0-\widetilde r_c))\sinh(2 \widetilde r_c)}{(2-c)\cosh^2(r_c)}\\
 & = \ln(1-\phi(T_1)) + \frac{2+c}{4}\left( 5\cosh (2a_0) + 3\cosh (2a_0 - 4\widetilde r_c) - 8 \cosh (2 \widetilde r_c)\right)\\
 &\leq\ln(1-\phi(T_1)) + 8  \cosh (2a_0).
\end{align*}

That is
 \begin{align}\label{phi-t-comp-1}
 \phi(t) \geq 1- (1-\phi(T_1))e^{8 \cosh (2a_0) }.
\end{align}
Since we have assumed $\phi(T_1) = 1-\epsilon$, from \eqref{phi-t-comp-1}, we have 
 \begin{align}\label{stay-lower-bound-1}
\phi(t) \geq\frac{1}{(1+\frac{\epsilon_1}{16})^2} > \frac{1}{(1+\frac{\epsilon_1}{8})^2} \quad \text{for all } t\in [T_1, T_2].
\end{align}
This contradicts the designation that $T_2$ is the first time that $\phi(T_2) = \frac{1}{(1+\frac{\epsilon_1}{8})^2}$. Therefore \eqref{theta-range-5} stays valid on $[T_1, T]$ for any $T>T_1$ unless possibly $\min_{S(t)}\Theta$ becomes larger than $\sqrt{1-\epsilon}$ (as long as the flow exists). Either case we conclude the positive lower bound for $\Theta$ as in \eqref{theta-bd}.
\end{proof}

\subsection{Proofs of Theorems \ref{main} and  \ref{main-thm-2}}
We now proceed to prove our main result on the MMCF in an {\afm}, namely Theorem \ref{main}. 
 \begin{proof}(of Theorem \ref{main}). The Lemma 
 \ref{keylemma} shows the MMCF exists smoothly with \eqref{theta-bd} for all time so that all higher order estimates are uniform in space and time. Now using the height estimates in Theorem \ref{height-1}, the flow remains in a compact region and thus converges smoothly to a stationary limiting surface. On the other hand, by Lemma \ref{mean-convex-1} we know that $u(\cdot, t)$ is monotonically decreasing since
   $\frac{\partial u(\cdot,t)}{\partial t}= -(H-c) \Theta \leq 0$, and thus $\lim\limits_{t\to \infty} \frac{\partial u(\cdot,t)}{\partial t} =0$.
We conclude that the limiting surface is of CMC 
with $c\in\left[0, 2-\frac{\epsilon_1}{2}\right]$.
\end{proof}

We can now prove the main application (Theorem \ref{main-thm-2}) of our analysis for the MMCF in a subclass of {\afm}s as in Theorem \ref{main}: the existence and uniqueness of the global monotone CMC foliation in such an {\afm}, confirming the Conjecture \ref{Thurston-Conj} in this case.
\begin{proof} (of Theorem \ref{main-thm-2}) For any fixed $c\in [0, 2-\frac{\epsilon_1}{2}]$, by Theorem \ref{main} and the height estimate Lemma \ref{height-1}, we have CMC surface $\Sigma_c = S_{\infty}^c$ lies between the two equidistant surfaces $\Sigma(r_c)$ and $\Sigma(\widetilde r_c)$.
Moreover, as a geodesic graph over $\Sigma$, by \eqref{theta-bd} we have
$$
\min_{\Sigma_c}\Theta\geq\frac{1}{1+\frac{\epsilon_1}{8}}>0.
$$
Since as $\epsilon_1 \to 0$ (and thus $\epsilon \to 0$ and $M^3$ tends to a Fuchsian manifold), we have $\widetilde r_c \to r_c = \tanh^{-1}\left(\frac{c}{2}\right)$ and the CMC surface $\Sigma_c$ is just the umbilical equidistant surface $\Sigma\left(r_c\right)$ with {\pc} $\frac{c}{2}<1$ (see Remark \ref{principle-curv-equidistant}) in the Fuchsian manifold $\Sigma\times \R$, we may choose a possibly smaller $\epsilon_1>0$ so that the two principal curvatures of $\Sigma_c$ at any point satisfies
$\max_{\Sigma_c}|\lambda| < 1$. This small {\pc} property implies that the equidistant surfaces $\Sigma_c(r)$ of $\Sigma_c$ foliate $M^3$ (see for instance \cite{Eps84}).

Recall that using Uhlenbeck's construction (Lemma \ref{U-metric}), the metric on $M^3$ (which is foliated by the equidistant surfaces
$\{\Sigma(r)\}_{r\in \R}$ of $\Sigma$)
$$
\begin{bmatrix}
1 & 0&0\\
0 & g_{11}(x,r) &  g_{12}(x,r)\\
0 &  g_{21}(x,r) &  g_{22}(x,r)
\end{bmatrix}
$$ 
satisfies:
\begin{equation}\label{metricODE}
    \frac{1}{2} \frac{\partial^2}{\partial r^2} g_{ij}(x,r) - \frac{1}{4} \frac{\partial }{\partial r} g_{il}(x,r)g^{lk}(x,r)\frac{\partial }{\partial r} g_{kj}(x,r) = g_{ij}(x,r)\,
\end{equation}
on $\Sigma \times \mathbb{R}$, subject to the initial data
$$
g(\cdot,0) = I_2 \quad \text{and}\quad \frac{1}{2} \frac{\partial }{\partial r} g (0) = A(\Sigma).
$$
Then the Uhlenbeck's metric for $M^3$ is uniquely determined as:
$$
g(\cdot, r) = g(r) = \left(\cosh(r) I_2 +\sinh(r) A(\Sigma_c)\right)^2,
$$
which is \textit{non-singular} for all $|r| < \tanh^{-1}(1)=\infty$. Since  the principal curvatures of $\Sigma_c$ at $(x,0)$ are $\lambda_1$ and $\lambda_2$ with $\lambda_1\geq \lambda_2, \max{|\lambda_i|}<1$ and $\lambda_1+\lambda_2 =c$, the principal curvatures of the equidistant surface $\Sigma_c(r)$ at $(x,r)$ are (c.f. \cite[Equation (3.10)]{Eps84}):
\begin{equation}
\frac{\lambda_1 \cosh(r) + \sinh(r)}{ \cosh(r) + \lambda_1 \sinh(r)} = \tanh\left(\tanh^{-1}(\lambda_1) + r\right)
\end{equation}
and
\begin{equation}
\frac{\lambda_2 \cosh(r) + \sinh(r)}{ \cosh(r) + \lambda_2 \sinh(r)}= \tanh\left(\tanh^{-1}(\lambda_2) + r\right),
\end{equation}
and thus the {\mc} of $\Sigma_c(r)$ at $(x,r)$ is
\begin{equation}
   H (x,r)= \frac{c\cosh(2r) + (1+\lambda_1\lambda_2)\sinh(2r)}{(\cosh(r)+\lambda_1\sinh(r))(\cosh(r)+\lambda_2\sinh(r))}.
\end{equation}
One easily checks that 
\begin{align*}
    \frac{\partial H(x,r)}{\partial r}
    =& \frac{1-\lambda_1^2}{(\cosh(r)+\lambda_1\sinh(r))^2}+ \frac{1-\lambda_2^2}{(\cosh(r)+\lambda_2\sinh(r))^2}>0,
\end{align*}
that is, the mean curvature of the equidistant surfaces $\Sigma_c(r)$ from $\Sigma_c$ is strictly increasing with respect to $r\in \R$.

Now with a direct application of the geometric {\maxp}  Theorem \ref{g-m-p} and arguing as in the proofs of \cite[Proposition 4.1, Lemma 4.3]{CMS23}, we will show the monotonicity of the CMC surfaces $\Sigma_c$'s with monotone values $c$'s and thus the uniqueness of $\Sigma_c$'s. For any $c \in [0, 2-\frac{\epsilon_1}{2}]$, since $\Sigma_c$ is closed we let
$$
H_{r,min} = \min_{x\in \Sigma_c}  \frac{\partial H(x,r)}{\partial r}>0 \quad \text{and}\quad H_{r,max} = \max_{x\in \Sigma_c}  \frac{\partial H(x,r)}{\partial r}>0
$$
which are clearly continuous with respect to $r$. Now denote
$$
g_{min}(r) = \int_{0}^r H_{s,min} ds  \quad \text{and}\quad g_{max}(r) = \int_{0}^r H_{s,max} ds  .
$$
Consider the hyperbolic signed distance of points on $\Sigma_{c'}$ to $\Sigma_c$ (where, without loss of generality, $c'>c$), 
$$
\max_{x\in \Sigma_{c'}} \text{dist} (x, \Sigma_c) = r_{max}=  \text{dist }  (p, \Sigma_c) \quad {\text{where }} p \in \Sigma_c(r_{max})
$$
and 
$$
\min_{x\in \Sigma_{c'}} \text{dist} (x, \Sigma_c) = r_{min}=  \text{dist }  (q, \Sigma_c) \quad {\text{where }} q \in \Sigma_c(r_{min})\,,
$$
then by the geometric maximum principle Theorem \ref{g-m-p} we know that
$$
 g_{min}(r_{max}) + c\leq H(p)\leq c'\leq H(q) \leq g_{max}(r_{min})+c
$$
and therefore
\begin{equation}\label{mono-1}
    g^{-1}_{min}(c'-c) \leq \text{dist }  (\Sigma_{c'}, \Sigma_c) \leq g^{-1}_{max}(c'-c).
\end{equation}
This yields the monotonicity of the CMC surfaces $\Sigma_c$'s with monotone values of $c$'s and thus also the uniqueness of $\Sigma_c$'s in $M^3$. Since \eqref{theta-bd} and therefore all higher order estimates hold uniformly for all $\Sigma_c$'s, \eqref{mono-1} also shows that $\Sigma_{c'}\to \Sigma_c$ smoothly as $c'\to c$. Since $c \in [0, 2-\frac{\epsilon_1}{2}]$ (similarly for $c\in [-2+\frac{\epsilon_1}{2},0]$ by symmetry) is arbitrary, combining this with Theorem \ref{main} and the foliation for the two ends in \cite{MP11} and \cite{CMS23}, it yields that $M^3$ admits a unique global monotone foliation by closed incompressible surfaces of constant mean curvature ranging from $-2$ to $2$.
\end{proof}
\bibliographystyle{amsalpha}
\bibliography{ref-mmcf}
\end{document}